\documentclass[preprint,12pt,number,sort&compress]{elsarticle}

\usepackage{amsfonts}
\usepackage{dsfont}
\usepackage{amssymb}
\usepackage{mathrsfs}
\usepackage{graphicx}
\usepackage{multirow}
\usepackage{color}
\usepackage{amsmath}
\usepackage{amsthm}
\usepackage{bm}
\usepackage{url}
\usepackage{ulem}
\usepackage{algorithm}
\usepackage{algorithmicx}
\usepackage{algpseudocode}
\newcommand{\br}{{\mathbb R}}

\newcommand{\bn}{{\mathbb N}}

\newcommand{\argmin}{{\rm argmin}}
\newcommand{\grad}{{\rm grad}}
\newcommand{\Exp}{{\rm Exp}}
\newcommand{\id}{{\rm id}}
\newcommand{\dom}{{\rm dom}}

\begin{document}

\begin{frontmatter}



\title{Convergence analysis of the Riemannian proximal gradient method with inexact oracle}


\author[1]{Chao Huang}
\ead{hchao@szu.edu.cn}
\author[2]{Xiyuan Xie}
\ead{xiexy65@mail2.sysu.edu.cn}
\author[3]{Anna Qi}
\ead{qiann@gpnu.edu.cn}
\author[2]{Lihua Yang}
\ead{mcsylh@mail.sysu.edu.cn}
\author[1]{Qian Zhang \corref{cor1}}
\ead{mazhangq@szu.edu.cn}

\address[1]{School of Mathematical Sciences, Shenzhen University, Shenzhen 518060, China}
\address[2]{School of Mathematics, Sun Yat-sen University, Guangzhou 510275, China}
\address[3]{School of Mathematics and Systems Science, Guangdong Polytechnic
Normal University, Guangzhou 510665, China}
\cortext[cor1]{Corresponding author.}

\begin{abstract}
We extend the notion of an inexact first-order oracle from the Euclidean setting to Riemannian optimization, and conduct a convergence analysis for
the Riemannian proximal gradient method equipped with this oracle, which we refer to as RPG-IO.
Under mild conditions on the oracle errors, we establish the global convergence of RPG-IO.
Specifically, we prove that (i) the norm of the search direction converges to zero;
(ii) the sequence of function values converges to the function value at any accumulation point of the iterates; and (iii) every accumulation point of the iterates is a stationary point.
Under the additional assumption of the Riemannian Kurdyka--{\L}ojasiewicz (KL) property, we
prove that the full sequence of iterates generated by RPG-IO converges to a single stationary
point. Moreover, we derive explicit convergence rates when the KL exponent is specified.
Finally, when a strong inexact oracle is employed, we establish the convergence rate of the sequence of function values to the optimal value.

\end{abstract}




\end{frontmatter}

\newtheorem{thm}{Theorem}[section]
\newtheorem{lem}[thm]{Lemma}
\newtheorem{de}{Definition}
\newtheorem{assump}{Assumption}
\newtheorem{example}{Example}
\newtheorem{coro}[thm]{Corollary}
\newtheorem{prop}[thm]{Proposition}
\newtheorem{alg}{Algorithm}
\newtheorem{eg}{Example}[section]
\renewcommand\abstractname{\bf Abstract}
\renewcommand\figurename{Figure}
\renewcommand\tablename{Table}
\renewcommand{\proofname}{\bf Proof}


\section{Introduction}
\setcounter{equation}{0}

Riemannian optimization, which aims to minimize an objective function defined on a Riemannian manifold, has attracted considerable attention in recent years
\cite{Absil-bk2009,Boumal-bk2023,Sato-bk2021,Wen-2020}.
In this paper, we consider a class of nonsmooth nonconvex optimization problems on Riemannian manifolds taking the following form:
\begin{equation}\label{eq:problem}
  \min_{x\in \mathcal{M}} F(x):=f(x) + h(x),
\end{equation}
where $\mathcal{M}$ is a finite-dimensional Riemannian manifold,
$f$ is a differentiable function (or more generally, a continuous function equipped with an inexact
first-order oracle), $h$ is continuous but possibly nonsmooth,
and $F$ is bounded below, i.e., there exists some $F^*>-\infty$ such that
$F(x)\ge F^*$ for all $x\in \mathcal{M}$.
This class of problems arises in a wide range of applications, including sparse principal component analysis,
compressed modes in physics, unsupervised feature selection, sparse blind deconvolution, and optimization with nonconvex regularizers;
see, e.g., \cite{Ma-2020} and references therein.

In the Euclidean setting (i.e., when $\mathcal{M}$ reduces to the Euclidean space $\br^n$),
the composite optimization problem \eqref{eq:problem} can be efficiently solved by the proximal
gradient method and its accelerated variants.
Starting from an initial point $x_0$, the method generates its iterates via\footnote{The commonly-used update expression is
$x_{k+1} = \argmin_{x\in\mathbb{R}^n} \langle \nabla f(x_k), x-x_k\rangle_\mathrm{F} + \frac{1}{2t}\|x-x_k\|_\mathrm{F}^2 + h(x)$.
We reformulate it equivalently for the convenience of the Riemannian formulation given later. }
\begin{equation}\label{eq:PGM}
  \begin{cases}
d_k = \argmin_{d\in\mathbb{R}^n} \langle \nabla f(x_k), d\rangle_\mathrm{F} + \frac{1}{2t}\|d\|_\mathrm{F}^2 + h(x_k + d), & \text{(Proximal mapping)} \\
x_{k+1} = x_k + d_k, & \text{(Update iterates)}
\end{cases}
\end{equation}
where $\langle \cdot,\cdot\rangle_\mathrm{F}$ denotes the Frobenius inner product:
$\langle u,v\rangle_\mathrm{F}:=u^\top v$ for vectors $u,v\in\br^n$ and $\langle u,v\rangle_\mathrm{F}:=\mathrm{tr}(u^\top v)$ for matrices $u,v\in\br^{n\times p}$.
The corresponding induced norm $\|u\|_\mathrm{F}:=\sqrt{\langle u,u\rangle_\mathrm{F}}$ is called the Frobenius norm.
Under the assumptions that $f$ is convex and $L$-smooth, $h$ is convex, and $F$ is coercive, the proximal
gradient method achieves an $\mathcal{O}(1/k)$ convergence rate for the objective value \cite{Beck-2009,Beck-bk2017}.

In the presence of manifold constraints, developing Riemannian proximal gradient methods is more challenging due to the nonlinearity of the underlying domain.
In \cite{Ma-2020}, Chen et al. proposed a proximal gradient method on the Stiefel manifold, called ManPG, as follows:
\begin{equation*}
\begin{cases}
\eta_k = \argmin_{\eta \in T_{x_k}\mathcal{M}} \langle \grad f(x_k), \eta \rangle_{\mathrm{F}} + \frac{1}{2t} \|\eta\|_{\mathrm{F}}^2 + h(x_k + \eta), \\
x_{k+1} = R_{x_k}(\alpha_k\eta_k),
\end{cases}
\end{equation*}
where $\grad f(x_k)$ denotes the Riemannian gradient of $f$ at $x_k$,
$T_{x_k}\mathcal{M}$ denotes the tangent space of the manifold $\mathcal{M}$ at $x_k$,
$R_{x_k}:T_{x_k}\mathcal{M}\rightarrow \mathcal{M}$ is the retraction map,
and $\alpha_k$ is the step size selected via a line search procedure.
It is shown that the subproblem at each iteration can be solved efficiently by a semi-smooth Newton method.
Furthermore, the global convergence of ManPG is established.

In \cite{Huang-2022-2}, Huang and Wei generalized the ManPG subproblem by introducing a weighting operator $W$ to replace the squared Frobenius norm
$\|\eta\|_{\mathrm{F}}^2$ with the weighted inner product $\langle \eta, W\eta \rangle_{\mathrm{F}}$.
In \cite{Huang-2022}, the same authors developed a Riemannian proximal gradient method, termed RPG,
which replaces the Euclidean addition $x_k + \eta$ with a retraction $R_{x_k}(\eta)$ and is thus well-defined on general manifolds.
In addition, the Riemannian metric $\langle \cdot,\cdot\rangle_x$ is used instead of the Frobenius inner product
$\langle\cdot,\cdot\rangle_{\text{F}}$, and a stationary point is used instead of a minimizer. More precisely, letting
\begin{equation*}
  \ell_{x_k}(\eta) := \langle \grad f(x_k), \eta \rangle_{x_k} + \frac{1}{2 t_k} \|\eta\|^2_{x_k} + h(R_{x_k}(\eta)),
\end{equation*}
the authors find $\eta_{x_k}^*\in T_{x_k}\mathcal{M}$ such that
\begin{equation*}
\eta_{x_k}^*\ \text{is a stationary point of}\ \ell_{x_k}(\eta)\
\text{in}\ T_{x_k}\mathcal{M}\ \text{and}\ \ell_{x_k}(\eta_{x_k}^*)\le \ell_{x_k}(0).
\end{equation*}
In contrast to ManPG, which only guarantees global convergence,
the local convergence rate of RPG has also been established under the Riemannian Kurdyka--{\L}ojasiewicz (KL) property.

In \cite{Choi-2025}, Choi et al. established the linear convergence rate of RPG and ManPG under the strong retraction convexity condition.
In \cite{Beck-2023}, Beck and Rosset addressed problem \eqref{eq:problem} over a compact embedded submanifold,
and proposed a dynamic smoothing gradient descent on manifolds (DSGM) algorithm
based on Riemannian gradient steps applied to a sequence of smooth approximations of the objective function.
In \cite{Bai-2023}, Bai and Bartoli developed a proxy step-size technique for proximal gradient on the sphere manifold by exploiting
the convexity and absolute homogeneity of $h(x)$,
which is easy to implement and much faster than the semi-smooth Newton method.

The standard convergence analysis of gradient-based methods relies on the availability of exact gradient information of
the objective function. However, in many optimization problems, one does not have
access to exact gradients, e.g., the gradient is obtained by solving an auxiliary optimization problem.
In this case, one can use inexact (approximate) gradient information.
Optimization algorithms with inexact first-order oracles are well studied in the literature, see, e.g., \cite{Nesterov-2014,Nabou-2025}.
In the Euclidean setting, it is well known that a convex and $L$-smooth function $f$ satisfies
\begin{equation}\label{eq:L-smooth}
  0 \leq f(x) - \left(f(y) + \left\langle \nabla f(y), x - y \right\rangle_\mathrm{F}\right) \leq \frac{L}{2}\|x - y\|_\mathrm{F}^2, \quad \forall\, x,y\in\br^n.
\end{equation}
In \cite{Nesterov-2014}, Devolder et al. considered the case where $h$ is the
indicator function of a convex set $Q$ and $f$ is a convex function, and proposed the definition of the
so-called inexact first-order $(\delta,L)$-oracle for $f$ by introducing a given amount $\delta$ of tolerance into inequality \eqref{eq:L-smooth}, that is, for any $y \in Q$, one can compute
an inexact oracle consisting of a pair $(f_{\delta,L}(y), g_{\delta,L}(y))$ such that
\begin{equation}\label{}
0 \leq f(x) - \left(f_{\delta,L}(y) + \left\langle g_{\delta,L}(y), x - y \right\rangle_\mathrm{F}\right) \leq \frac{L}{2}\|x - y\|_\mathrm{F}^2 + \delta, \quad \forall\, x \in Q.
\end{equation}
Recently, Nabou et al. \cite{Nabou-2025} introduced a definition of inexactness for a first-order oracle for
$f$ involving a degree $q\in[0,2)$, that is,
\begin{equation}\label{eq:Nabou-inexact-oracle}
 f(x) - \left(f(y) + \left\langle g_{\delta,L}(y), x - y \right\rangle_\mathrm{F}\right)
 \leq \frac{L}{2}\|x - y\|_\mathrm{F}^2 + \delta\|x - y\|_\mathrm{F}^q, \quad \forall x\, \in \dom(f).
\end{equation}

Riemannian optimization with inexact first-order oracles has also been well studied in the literature.
In \cite{Huang-2023}, Huang and Wei proposed an inexact Riemannian proximal gradient method
named IRPG, which avoids the need to exactly solve the RPG subproblem, and established its global convergence and local convergence rates.
In \cite{Deng-2023}, Deng and Peng developed an inexact augmented Lagrangian framework on manifolds for solving problem \eqref{eq:problem}.
By leveraging the Moreau envelope, the authors obtained a smoothed Riemannian minimization subproblem, which is solved via a Riemannian Barzilai--Borwein gradient method.
In \cite{Deng-2025-2}, Deng et al. proposed two inexact augmented Lagrangian methods on manifolds, namely ManIAL and StoManIAL, with provable oracle complexity guarantees.
In \cite{Deng-2025}, Zhou et al. developed a unified algorithmic framework for inexact gradient methods on Riemannian manifolds,
and established strong convergence results under two types of inexactness conditions and the Riemannian KL property.

In this paper, we extend the notion of an inexact first-order oracle from the Euclidean setting to Riemannian optimization, and conduct a rigorous convergence analysis of
the Riemannian proximal gradient method equipped with such an inexact oracle (RPG-IO).
The main contributions of this work are summarized as follows:
\begin{itemize}
  \item We define an inexact first-order oracle tailored to Riemannian optimization,
  and present representative examples as well as a fundamental property of the proposed oracle framework.
  \item RPG-IO is not a descent method due to oracle errors. Nevertheless, we establish its global convergence:
the norm of the search direction converges to zero;
the sequence of function values converges to the function value at any accumulation point of the iterates;
and every accumulation point of the iterates is a stationary point.
  \item Under the additional assumption of the Riemannian KL property, we
prove that the full sequence of iterates generated by RPG-IO converges to a single stationary
point. Moreover, we derive explicit convergence rates when the KL exponent is specified.
To the best of our knowledge, the convergence rate results derived in this work differ not only from classical Euclidean results (e.g., \cite{Bolte-2009,Li-2023})
but also from existing results in Riemannian optimization (e.g., \cite{Huang-2022,Huang-2023,Deng-2025}).
\item When a strong inexact oracle is employed, we further establish the convergence rate of the sequence of function values to the optimal value.
\end{itemize}

The remainder of this paper is organized as follows.
Notation and preliminaries on Riemannian manifolds are presented in Section 2.
In Section 3, we define an inexact first-order oracle for Riemannian optimization,
together with illustrative examples and a key property of the proposed oracle.
In Section 4, we analyze the convergence of the Riemannian proximal gradient method equipped with this oracle,
covering global convergence, as well as sequential convergence and convergence rates under the Riemannian KL property.
In Section 5, we define a strong inexact oracle and establish the
convergence rate of the sequence of function values to the optimal value.
Finally, Section 6 concludes the paper with a brief summary.

\section{Notation and preliminaries on Riemannian manifolds}
\setcounter{equation}{0}

We follow standard references for the Riemannian concepts used in this paper; see, e.g., \cite{Absil-bk2009,Boumal-bk2023,Lee-book2018,Carmo-book1992}.
A Riemannian manifold $\mathcal{M}$ is a manifold
endowed with a Riemannian metric $(\xi_x, \eta_x) \mapsto \langle \xi_x, \eta_x \rangle_x \in \mathbb{R}$, where $\xi_x$ and $\eta_x$ are
tangent vectors in the tangent space at $x$. The induced norm on the tangent space
at $x$ is denoted by $\| \cdot \|_x$ or $\| \cdot \|$ when the subscript is clear from context. The
tangent space to the manifold $\mathcal{M}$ at $x$ is denoted by $T_x \mathcal{M}$, and the tangent bundle,
which is the set of all tangent vectors, is denoted by $T \mathcal{M}$. A vector field is a function
from the manifold to its tangent bundle, i.e., $\eta: \mathcal{M} \to T \mathcal{M}: x \mapsto \eta_x$.

The Riemannian gradient of a function $f: \mathcal{M} \to \br$, denoted by $\grad f(x)$, is the unique tangent vector satisfying
\begin{equation*}
Df(x)[\eta] = \langle \grad f(x), \eta\rangle_x, \quad \forall\, \eta \in T_x \mathcal{M},
\end{equation*}
where $Df(x)[\eta]$ denotes the directional derivative of $f$ at $x \in \mathcal{M}$ in the direction $\eta\in T_x \mathcal{M}$.

A smooth map $R: T\mathcal{M} \to \mathcal{M}$ is called a retraction on a manifold $\mathcal{M}$ if its restriction to $T_x\mathcal{M}$,
denoted by $R_x: T_x\mathcal{M} \to \mathcal{M}$, satisfies
\begin{itemize}
    \item $R_x(0_x) = x$ for all $x \in \mathcal{M}$, where $0_x$ denotes the zero element of $T_x\mathcal{M}$;
    \item The differential of $R_x$ at $0_x$ is the identity map on $T_x\mathcal{M}$, i.e., $DR_x(0_x) = \mathrm{id}_{T_x\mathcal{M}}$.
\end{itemize}
The most natural retraction is the exponential map, denoted by $\text{Exp}$, which satisfies $\mathrm{Exp}_x(\eta_x) = \gamma(1)$, where $\gamma$ is the geodesic
with $\gamma(0) = x$ and $\gamma'(0) = \eta_x$.
Unfortunately, the exponential map is
itself defined as the solution to a nonlinear ordinary differential equation, which generally
poses significant numerical challenges for efficient computation. Therefore, we consider
alternative retractions that can be computed more efficiently
\cite{Absil-bk2009,Boumal-bk2023,Wen-2020}.

A smooth map $T\mathcal{M} \times T\mathcal{M} \to T\mathcal{M}: (\eta_x, \xi_x) \mapsto \mathcal{T}_{\eta_x}(\xi_x)$
is called a vector transport associated with
a retraction $R$ on a manifold $\mathcal{M}$ if it satisfies the following conditions:
\begin{itemize}
    \item $\mathcal{T}_{\eta_x}: T_x\mathcal{M} \rightarrow T_{R_x(\eta_x)}\mathcal{M}$ is a linear map;
    \item $\mathcal{T}_{0_x}\xi_x = \xi_x$ for all $\xi_x \in T_x\mathcal{M}$.
\end{itemize}
The adjoint operator $\mathcal{T}^\sharp$ of a vector transport $\mathcal{T}$ is defined by the relation $\langle \xi_y, \mathcal{T}_{\eta_x}\zeta_x \rangle_y = \langle \mathcal{T}_{\eta_x}^\sharp \xi_y, \zeta_x \rangle_x$ for all $\eta_x, \zeta_x \in T_x\mathcal{M}$ and $\xi_y \in T_y\mathcal{M}$, where $y = R_x(\eta_x)$.
An important vector transport is parallel transport.
We refer the reader to \cite{Absil-bk2009,Boumal-bk2023,Lee-book2018,Carmo-book1992} for its rigorous definition.
Parallel transport is not the only way to define a vector transport.
In fact, there is considerable flexibility in how a vector transport is chosen for a given problem.
In this paper, we employ a particularly tractable class of vector transports, namely, the vector transport by differentiated retraction, given by $\mathcal{T}_{\eta_x}:=DR_x(\eta_x)$ \cite{Absil-bk2009}.

The generalized subdifferential on Euclidean spaces is defined in \cite{Clarke-book1990}.
Recall that if $G$ is a locally Lipschitz function from a Hilbert space $X$ to $\mathbb{R}$, the Clarke generalized directional derivative of $G$ at $x \in X$ in the direction $v \in X$, denoted by $G^\circ(x;v)$, is defined by
$$
G^\circ(x;v) = \limsup_{y \to x,\ t \downarrow 0} \frac{G(y + tv) - G(y)}{t},
$$
and the Clarke generalized subdifferential of $G$ at $x$, denoted by $\partial G(x)$, is defined by
$$
\partial G(x) := \{\xi \in X : \langle \xi, v \rangle \leq G^\circ(x;v) \text{ for all } v \in X\}.
$$

The Riemannian version of the generalized subdifferential is introduced in \cite{Hosseini-2018}.
Let $f : \mathcal{M} \to \mathbb{R}$ be a locally Lipschitz function defined on a Riemannian manifold $\mathcal{M}$. For $x \in \mathcal{M}$, let $\hat{f}_x = f \circ R_x$ denote the restriction of the pullback $\hat{f} = f \circ R$ to $T_x \mathcal{M}$. The Riemannian version of the Clarke generalized directional derivative of $f$ at $x$ in the direction $\eta \in T_x \mathcal{M}$, denoted by $f^\circ(x;\eta)$, is defined as $f^\circ(x;\eta) := \hat{f}_x^\circ(0_x;\eta)$, where $\hat{f}_x^\circ(0_x;\eta)$ denotes the Clarke generalized directional derivative of $\hat{f}_x : T_x \mathcal{M} \to \mathbb{R}$ at $0_x$ in the direction $\eta \in T_x \mathcal{M}$. Correspondingly, the Riemannian subdifferential of $f$ at $x$, denoted by $\partial f(x)$, is defined as $\partial f(x) := \partial \hat{f}_x(0_x)$.
Any element of $\partial f(x)$ is called a Riemannian subgradient of $f$ at $x$.
Note that other equivalent definitions of the Clarke generalized directional derivative and generalized subdifferential exist for functions on Riemannian manifolds \cite{Hosseini-2011}. A point $x$ is called a stationary point of $f$ if $0 \in \partial f(x)$. A necessary condition for $f$ to attain a local minimum at $x$ is that $x$ is a stationary point of $f$ \cite{Hosseini-2011,Yang-2014}.

The injectivity radius of a Riemannian manifold $\mathcal{M}$ is defined as $\mathrm{inj}(\mathcal{M}):=\inf_{x\in \mathcal{M}}\mathrm{inj}(x)$,
where $\mathrm{inj}(x):=\sup\{\varepsilon>0:\Exp_x|_{B(0_x,\,\varepsilon)}\ \text{is a diffeomorphism}\}$.
Note that any compact Riemannian manifold has positive injectivity radius \cite[Lemma 6.16]{Lee-book2018}.
Suppose that $\mathcal{M}$ has positive injectivity radius.
A vector field $\eta$ is called Lipschitz continuous if there exists a constant $\tilde{L}_f > 0$ such that for any $x,y\in \mathcal{M}$ with $\mathrm{dist}(x,y)<\mathrm{inj}(\mathcal{M})$, there holds \cite{Huang-2022}
\begin{equation}\label{de:Lipschitz}
  \|\eta_x - \mathrm{P}_{y\to x} \eta_y\| \le \tilde{L}_f\,\mathrm{dist}(x,y),
\end{equation}
where $\mathrm{P}_{y\to x}:T_y\mathcal{M}\to T_x\mathcal{M}$ denotes the parallel transport along the unique minimizing geodesic joining $y$ to $x$.
A vector field $\eta$ is called locally Lipschitz continuous if for any compact subset $\bar{\Omega}$ of $\mathcal{M}$,
there exists a constant $\tilde{L}_f > 0$ such that for any $x,y\in \bar{\Omega}$ with $\mathrm{dist}(x,y)<\mathrm{inj}(\bar{\Omega})$,
inequality \eqref{de:Lipschitz} holds.
A function on $\mathcal{M}$ is called (locally) Lipschitz continuously differentiable if the vector field of its Riemannian gradient
is (locally) Lipschitz continuous.

\section{ Inexact first-order oracle for Riemannian optimization}
\setcounter{equation}{0}

The following concept of $L$-retraction-smoothness is the generalization of
$L$-smoothness from Euclidean spaces to the Riemannian manifold setting.
Boumal et al. \cite{Boumal-2018} first adopted this concept to analyze the convergence of optimization algorithms on Riemannian manifolds.
Recently, Huang and Wei \cite{Huang-2022} employed this concept to investigate the Riemannian proximal gradient method.

\begin{de}\label{de-L-retraction-smooth}
\cite{Boumal-2018,Huang-2022} A function $f:\mathcal{M}\rightarrow \br$ is called
$L_f$-retraction-smooth with respect to a retraction $R$ if for any $x\in \mathcal{M}$
and any $\eta\in T_x\mathcal{M}$, there holds
\begin{equation}\label{eq:retraction-smooth}
  f(R_x(\eta))\le f(x) + \langle \grad f(x),\eta\rangle + \frac{L_f}{2}\|\eta\|^2.
\end{equation}
\end{de}

The $(\delta,L)$-type inexact first-order oracle was originally formalized by Devolder et al. \cite{Nesterov-2014} in Euclidean spaces.
Recently, Nabou et al. \cite{Nabou-2025} generalized this definition by introducing an error term of order $q$, as formally defined in \eqref{eq:Nabou-inexact-oracle}. By combining this generalized inexact oracle framework with the notion of $L$-retraction-smoothness on Riemannian manifolds, we define an inexact first-order oracle for Riemannian optimization.

\begin{de}\label{de-inexact-oracle}
A function $f:\mathcal{M}\rightarrow \br$ is said to be equipped with an inexact first-order $(\delta,L)$-oracle of degree $q\in[0,2)$
if for any $x\in \mathcal{M}$, there exists $g_{\delta,L,q}(x)\in T_x\mathcal{M}$ such that
\begin{equation}\label{eq:inexact-oracle}
  f(R_x(\eta))\le f(x) + \langle g_{\delta,L,q}(x),\eta\rangle + \frac{L}{2}\|\eta\|^2 + \delta\|\eta\|^q,~~~~\forall\,\eta\in T_x\mathcal{M}.
\end{equation}
\end{de}

Next, we present several illustrative examples that satisfy Definition \ref{de-inexact-oracle}.

\noindent\textbf{Example 1.} (Smooth function with an approximate Riemannian gradient)
Let $f:\mathcal{M}\rightarrow \br$ be $L_f$-retraction-smooth.
Assume that for any $x\in \mathcal{M}$, there exists $g_{\epsilon,L_f}(x)\in T_x\mathcal{M}$ satisfying
\begin{equation*}
  \|g_{\epsilon,L_f}(x)-\grad f(x)\| \le \epsilon.
\end{equation*}
Then $f$ is equipped with a $(\delta,L)$-oracle of degree $q=1$ in the sense of Definition \ref{de-inexact-oracle},
with $g_{\delta,L,1}(x)=g_{\epsilon,L_f}(x)$, $\delta = \epsilon$, and $L=L_f$. To verify this claim, note that
$$
\begin{aligned}
f(R_x(\eta))
&\leq f(x) + \langle \grad f(x), \eta \rangle + \frac{L_f}{2}\|\eta\|^2 \\
&= f(x) + \langle g_{\epsilon,L_f}(x), \eta \rangle + \langle \grad f(x) - g_{\epsilon,L_f}(x), \eta \rangle + \frac{L_f}{2}\|\eta\|^2 \\
&\leq f(x) + \langle g_{\epsilon,L_f}(x), \eta \rangle + \frac{L_f}{2}\|\eta\|^2 + \epsilon\|\eta\|.
\end{aligned}
$$

\noindent\textbf{Example 2.} (Gradient evaluation at nearby points)
Assume that $\mathcal{M}$ has positive injectivity radius.
Suppose that $f:\mathcal{M}\rightarrow \br$ is $L_f$-retraction-smooth and Lipschitz continuously differentiable.
Thus, there exists a constant $\tilde{L}_f>0$ such that for any $x,y\in \mathcal{M}$
with $\mathrm{dist}(x,y)<\mathrm{inj}(\mathcal{M})$, there holds
\begin{equation*}
  \|\grad f(x) - \mathrm{P}_{y\to x} \grad f(y)\| \le \tilde{L}_f\,\mathrm{dist}(x,y),
\end{equation*}
where $\mathrm{P}_{y\to x}:T_y\mathcal{M}\to T_x\mathcal{M}$ denotes the parallel transport along the unique minimizing geodesic joining $y$ to $x$.
For any $x\in \mathcal{M}$, we assume that we have access to the exact Riemannian gradient
only at a nearby point $\hat{x}\neq x$ that satisfies $\mathrm{dist}(x,\hat{x})\le \Delta < \mathrm{inj}(\mathcal{M})$.
Then $f$ is equipped with a $(\delta,L)$-oracle of degree $q=1$ in the sense of Definition \ref{de-inexact-oracle},
with $g_{\delta,L,1}(x)=\mathrm{P}_{\hat{x}\to x}\grad f(\hat{x})$, $\delta = \tilde{L}_f\Delta$, and $L=L_f$. To verify this claim, for any $\eta\in T_x\mathcal{M}$, we have
$$
\begin{aligned}
f(R_x(\eta))
&\leq f(x) + \langle \grad f(x), \eta \rangle + \frac{L_f}{2}\|\eta\|^2 \\
&= f(x) + \langle \mathrm{P}_{\hat{x}\to x}\grad f(\hat{x}), \eta \rangle + \langle \grad f(x) - \mathrm{P}_{\hat{x}\to x}\grad f(\hat{x}), \eta \rangle + \frac{L_f}{2}\|\eta\|^2 \\
&\leq f(x) + \langle \mathrm{P}_{\hat{x}\to x}\grad f(\hat{x}), \eta \rangle + \frac{L_f}{2}\|\eta\|^2 + \tilde{L}_f\Delta\|\eta\|.
\end{aligned}
$$

As established in \cite{Nabou-2025}, for any $\rho>0$, there holds
\begin{equation}\label{eq:Nabou}
\delta \|\eta\|^q \leq \frac{q\rho}{2} \|\eta\|^2 + \frac{2-q}{2\rho^{\frac{q}{2-q}}} \delta^{\frac{2}{2-q}}.
\end{equation}
Consequently, by Definition \ref{de-inexact-oracle}, if a function $f:\mathcal{M}\rightarrow \br$ admits an inexact first-order $(\delta,L)$-oracle of degree $q\in[0,2)$,
there holds
\begin{equation}\label{eq:inexact-inequality}
f(R_x(\eta)) \leq f(x) + \langle g_{\delta,L,q}(x), \eta \rangle + \frac{L + q\rho}{2} \|\eta\|^2 + \frac{2-q}{2\rho^{\frac{q}{2-q}}} \delta^{\frac{2}{2-q}}.
\end{equation}

To study the convergence of optimization algorithms on Riemannian manifolds, we introduce the following notion of
convexity on manifolds.

\begin{de}\label{de-retraction-convex}
\cite{Huang-2022} A function $f:\mathcal{M}\rightarrow \br$ is called
retraction-convex with respect to a retraction $R$ if for any $x\in \mathcal{M}$,
any $\xi,\eta\in T_x\mathcal{M}$, and any $\zeta\in \partial(f\circ R_x)(\xi)$, there holds
\begin{equation}\label{eq:retraction-convex}
  f(R_x(\eta)) \ge f(R_x(\xi)) + \langle \zeta, \eta-\xi\rangle.
\end{equation}
\end{de}

In Definition \ref{de-inexact-oracle}, $g_{\delta,L,q}(x)$ can be regarded as an approximation to the exact Riemannian gradient (or Riemannian subgradient) of
$f$. The following property provides an upper bound on the magnitude of such approximation errors.

\begin{prop}\label{prop}
Let a retraction-convex function $f:\mathcal{M}\rightarrow \br$ admit a $(\delta,L)$-oracle of degree $q\in[0,2)$.
Then for any $x\in \mathcal{M}$ and any Riemannian subgradient $g(x)\in \partial f(x)$, the following statements hold:
\begin{description}
  \item[(1)] If $q\in [0,1)$, then
  \begin{equation}\label{}
    \|g(x)-g_{\delta,L,q}(x)\| \le (2-q)\left(\frac{L}{2(1-q)}\right)^{\frac{1-q}{2-q}}\delta^{\frac{1}{2-q}}.
  \end{equation}
  \item[(2)] If $q=1$, then
  \begin{equation}\label{}
    \|g(x)-g_{\delta,L,q}(x)\| \le \delta.
  \end{equation}
  \item[(3)] If $q\in(1,2)$, then $f$ is differentiable,
  and $g_{\delta,L,q}(x) = \grad f(x)$.
\end{description}
\end{prop}
\begin{proof}
By a standard result in \cite{Hosseini-2011} and the definition of a retraction,
we have
\begin{equation*}
\partial (f \circ R_x)(0) = [D R_x(0)]^\sharp \partial f(R_{x}(0)) = \mathrm{id}_{T_x\mathcal{M}}^\sharp \partial f(x) = \partial f(x).
\end{equation*}
Since $f$ is retraction-convex with respect to the retraction $R$ and $g(x)\in \partial f(x)$, it follows that
\begin{equation*}
f(R_x(\eta)) \geq f(x) + \langle g(x),\eta\rangle.
\end{equation*}
Combining this with inequality \eqref{eq:inexact-inequality}, we obtain
\begin{equation*}
 \langle g(x) - g_{\delta,L,q}(x), \eta \rangle \leq   \frac{L + q\rho}{2} \|\eta\|^2 + \frac{2-q}{2\rho^{\frac{q}{2-q}}} \delta^{\frac{2}{2-q}},
 ~~~~\forall\,\eta\in T_x\mathcal{M}.
\end{equation*}
Taking $\eta = t\frac{g(x) - g_{\delta,L,q}(x)}{\|g(x) - g_{\delta,L,q}(x)\|}$ with $t>0$ in the above inequality, we get
\begin{equation*}
 t\| g(x) - g_{\delta,L,q}(x) \| \leq   \frac{L + q\rho}{2} t^2 + \frac{2-q}{2\rho^{\frac{q}{2-q}}} \delta^{\frac{2}{2-q}}.
\end{equation*}
Dividing both sides by $t>0$ yields
\begin{equation*}
 \| g(x) - g_{\delta,L,q}(x) \| \leq   \frac{L + q\rho}{2} t + \frac{1}{t}\cdot\frac{2-q}{2\rho^{\frac{q}{2-q}}} \delta^{\frac{2}{2-q}}.
\end{equation*}
This upper bound attains its minimum $\sqrt{\frac{(2-q)(L+q\rho)}{\rho^{\frac{q}{2-q}}}}\,\delta^{\frac{1}{2-q}}$ when
$t=\sqrt{\frac{2-q}{\rho^{\frac{q}{2-q}}(L+q\rho)}}\,\delta^{\frac{1}{2-q}}$.
Thus, for any $\rho>0$, we have
\begin{equation}\label{eq:prop}
  \|g(x)-g_{\delta,L,q}(x)\| \le \sqrt{\frac{(2-q)(L+q\rho)}{\rho^{\frac{q}{2-q}}}}\,\delta^{\frac{1}{2-q}}.
\end{equation}
Finally, taking the infimum of the right-hand side over $\rho>0$, we complete the proof.
\end{proof}

Proposition \ref{prop} requires $f$
to be retraction-convex. Nevertheless, the conclusions of Proposition \ref{prop} may still hold even without the retraction-convex condition,
as shown in Examples 1 and 2 above. In light of this observation, we now introduce the following assumption.

\begin{assump}\label{assump:bound}
Suppose that $f$ admits a $(\delta_k,L_k)$-oracle of degree $q\in[0,2)$ for all $k\ge 0$, that $\{L_k\}$ is uniformly bounded above,
and that there exists a constant $C_{q}\ge0$ depending only on $q$ such that for all $k\ge0$, any $x\in \mathcal{M}$,
and any Riemannian subgradient $g(x)\in \partial f(x)$, we have
\begin{equation}\label{}
  \|g(x)-g_{\delta_k,L_k,q}(x)\| \le C_{q}\delta_k^{\frac{1}{2-q}}.
\end{equation}
\end{assump}

\section{Convergence analysis of the Riemannian proximal gradient method with inexact oracle}
\setcounter{equation}{0}

The Riemannian proximal gradient method with inexact oracle (RPG-IO) proposed in this paper is presented in Algorithm \ref{alg:algorithm-IRPG}.
In each iteration, the algorithm first computes a search direction by solving
a proximal subproblem in the tangent space at the current iterate, and then obtains a new iterate via retraction.
Unlike the RPG algorithm proposed by Huang and Wei in \cite{Huang-2023},
our algorithm uses the approximate Riemannian gradient $g_{\delta_k, L_k, q}(x_k)$
instead of the exact Riemannian gradient $\grad f(x_k)$ at each iteration.

\begin{algorithm}[htb]
\caption{ Riemannian proximal gradient method with inexact oracle (RPG-IO)}
\label{alg:algorithm-IRPG}
\begin{algorithmic}[1]
\Require An initial point $x_0\in\mathcal{M}$; the degree parameter $q\in[0,2)$
\For{$k=0,1,\cdots$}
%\While{$E(x^{k-1})-E(x^{k})>\epsilon$}
    \State  Choose $\delta_k\ge0$, $L_k\ge0$, and $t_k>0$; then obtain $g_{\delta_k, L_k, q}(x_k)$.
    \State  Let $\ell_{x_k}(\eta) = \langle g_{\delta_k, L_k, q}(x_k), \eta \rangle + \frac{1}{2 t_k} \|\eta\|^2 + h(R_{x_k}(\eta))$.
    \State  Find $\eta_{x_k}^*\in T_{x_k}\mathcal{M}$ such that
            \begin{equation*}
              \eta_{x_k}^*\ \text{is a stationary point of}\ \ell_{x_k}(\eta)\
              \text{in}\ T_{x_k}\mathcal{M},\ \text{and}\ \ell_{x_k}(\eta_{x_k}^*)\le \ell_{x_k}(0).
            \end{equation*}
    \State  Set $x_{k+1} = R_{x_k}(\eta_{x_k}^*)$.
    \EndFor
\end{algorithmic}
\end{algorithm}

\subsection{Global convergence analysis}

We first derive a recurrence for the sequence $\{F(x_k)\}$,
which serves as the basis for our subsequent analysis.

\begin{thm}\label{thm:IRPG-decrease}
Suppose that for each iteration $k\ge0$, the function $f$ admits a $(\delta_k, L_k)$-oracle of degree $q \in [0, 2)$ with $\delta_k \geq 0$ and $L_k \geq 0$.
Let the sequence $\{x_k\}$ be generated by Algorithm \ref{alg:algorithm-IRPG}.
\begin{description}
  \item[(1)] If there exists a constant $\rho>0$ such that $t_k \leq \frac{1}{2(L_k + q\rho)}$ for all $k\ge0$, then we have
\begin{equation}\label{eq:decrease-1}
F(x_{k+1}) \leq F(x_k) - \frac{1}{4t_k} \|\eta_{x_k}^*\|^2  + \frac{2-q}{2\rho^{\frac{q}{2-q}}} \delta_k^{\frac{2}{2-q}}.
\end{equation}
  \item[(2)] If there exists a constant $\rho>0$ such that $t_k \leq \frac{1}{L_k + q\rho}$ for all $k\ge0$, and the function $h$ is retraction-convex,
  then we have
\begin{equation}\label{eq:decrease-2}
  F(x_{k+1}) \leq F(x_k) - \frac{1}{2t_k} \|\eta_{x_k}^*\|^2  + \frac{2-q}{2\rho^{\frac{q}{2-q}}} \delta_k^{\frac{2}{2-q}}.
\end{equation}
\end{description}
\end{thm}
\begin{proof}
Let $g_k=g_{\delta_k, L_k, q}(x_k)$.
Using inequality \eqref{eq:inexact-inequality}, we obtain
\begin{equation*}
  f(x_{k+1}) = f(R_{x_k}(\eta_{x_k}^*)) \leq f(x_k) + \langle g_k, \eta_{x_k}^* \rangle + \frac{L_k + q\rho}{2} \|\eta_{x_k}^*\|^2 + \frac{2-q}{2\rho^{\frac{q}{2-q}}} \delta_k^{\frac{2}{2-q}}.
\end{equation*}

\noindent(1) By the definition of $\eta_{x_k}^*$, we have
$$
\begin{aligned}
F(x_k) &\geq f(x_k) + \langle g_k, \eta_{x_k}^*\rangle + \frac{1}{2t_k}\|\eta_{x_k}^*\|^2 + h(R_{x_k}(\eta_{x_k}^*)) \\
&\geq f(R_{x_k}(\eta_{x_k}^*)) + h(R_{x_k}(\eta_{x_k}^*)) +\frac{1}{2}\left(\frac{1}{t_k}-(L_k+q\rho)\right)\|\eta_{x_k}^*\|^2
 - \frac{2-q}{2\rho^{\frac{q}{2-q}}} \delta_k^{\frac{2}{2-q}} \\
&\geq F(x_{k+1}) + \frac{1}{4t_k}\|\eta_{x_k}^*\|^2 - \frac{2-q}{2\rho^{\frac{q}{2-q}}} \delta_k^{\frac{2}{2-q}},
\end{aligned}
$$
which implies inequality \eqref{eq:decrease-1}.

\noindent(2) From the optimality condition of the subproblem in Algorithm \ref{alg:algorithm-IRPG}, we have
$$0 \in g_k + \frac{1}{t_k} \eta_{x_k}^* + \partial (h \circ R_{x_k})(\eta_{x_k}^*). $$
As shown in \cite{Hosseini-2011}, $\partial (h \circ R_{x_k})(\eta_{x_k}^*) = [D R_{x_k}(\eta_{x_k}^*)]^\sharp \partial h(R_{x_k}(\eta_{x_k}^*))$.
Since $D R_{x_k}(\eta_{x_k}^*) = \mathcal{T}_{\eta_{x_k}^*}$, it follows that
$\partial (h \circ R_{x_k})(\eta_{x_k}^*) = \mathcal{T}_{\eta_{x_k}^*}^\sharp \partial h(x_{k+1})$.
Thus there exists $\zeta_{x_{k+1}} \in \partial h(x_{k+1})$ such that
\begin{equation}\label{eq:eq1}
g_k + \frac{1}{t_k} \eta_{x_k}^* + \mathcal{T}_{\eta_{x_k}^*}^\sharp \zeta_{x_{k+1}} = 0.
\end{equation}
Since $h$ is retraction-convex with respect to the retraction $R$, there holds
\begin{equation}\label{eq:eq2}
h(x_k) - h(x_{k+1}) = h(R_{x_k}(0)) - h(R_{x_k}(\eta_{x_k}^*))\geq -\langle \mathcal{T}_{\eta_{x_k}^*}^\sharp \zeta_{x_{k+1}},\ \eta_{x_k}^* \rangle.
\end{equation}
We deduce that
$$
\begin{aligned}
f(x_{k+1}) &\leq f(x_k) + \langle g_k, \eta_{x_k}^* \rangle + \frac{L_k + q\rho}{2} \|\eta_{x_k}^*\|^2 + \frac{2-q}{2\rho^{\frac{q}{2-q}}} \delta_k^{\frac{2}{2-q}} \\
%&= f(x_k) + \langle g_k + \mathcal{T}_{\eta_{x_k}^*}^\sharp \zeta_{x_{k+1}}, \eta_{x_k}^* \rangle - \langle \mathcal{T}_{\eta_{x_k}^*}^\sharp \zeta_{x_{k+1}}, \eta_{x_k}^* \rangle
% + \frac{L_k + q\rho}{2} \|\eta_{x_k}^*\|^2 + \frac{2-q}{2\rho^{\frac{q}{2-q}}} \delta_k^{\frac{2}{2-q}} \\
&\leq f(x_k) - \frac{1}{t_k} \|\eta_{x_k}^*\|^2 + h(x_k) - h(x_{k+1}) + \frac{L_k + q\rho}{2} \|\eta_{x_k}^*\|^2 + \frac{2-q}{2\rho^{\frac{q}{2-q}}} \delta_k^{\frac{2}{2-q}} \\
&\leq f(x_k) - \frac{1}{2t_k} \|\eta_{x_k}^*\|^2 + h(x_k) - h(x_{k+1}) + \frac{2-q}{2\rho^{\frac{q}{2-q}}} \delta_k^{\frac{2}{2-q}},
\end{aligned}
$$
where the second inequality follows from \eqref{eq:eq1} and \eqref{eq:eq2}.
Then we obtain inequality \eqref{eq:decrease-2}.
\end{proof}

\noindent
\textbf{Remark:}
The recurrence inequality established in Theorem \ref{thm:IRPG-decrease} does not necessarily imply a descent property owing to the error term.
Nevertheless, it reveals an approximate descent property as shown in \eqref{eq:approximate descent},
which allows us to characterize the descent behavior of the algorithm in an alternative way.

The global convergence analysis of Algorithm \ref{alg:algorithm-IRPG} requires the following lemma.

\begin{lem}\label{lem-Hosseini}
\cite{Hosseini-2018} If $\{x_k\}$ and $\{\xi_k\}$ are sequences in $\mathcal{M}$ and $T\mathcal{M}$, respectively, such that $\xi_k \in \partial f(x_k)$ for each $k$,
and if $x_k\rightarrow x$, $f(x_k)\rightarrow f(x)$, and $\xi_k\rightarrow \xi$ as $k\rightarrow\infty$, then
$\xi \in \partial f(x)$.
\end{lem}

The following theorem demonstrates that, under mild conditions on the oracle errors,
the global convergence of Algorithm \ref{alg:algorithm-IRPG} can be established.
The convergence results include: the norm of the search direction converges to zero,
the sequence of function values converges to the function value at any accumulation point of the iterates, and any accumulation point of the iterates is a stationary point.

\begin{thm}\label{thm:IRPG-convergence}
Suppose the conditions of Theorem \ref{thm:IRPG-decrease} hold,
and suppose further that
\begin{equation}\label{}
\sum_{k=0}^{\infty} \delta_k^{\frac{2}{2-q}}<\infty.
\end{equation}
Then the following convergence results hold:
\begin{description}
  \item[(1)] $\lim_{k\rightarrow\infty}\|\eta_{x_k}^*\|=0$.
  \item[(2)] If $\bar{x}$ is an accumulation point of the sequence $\{x_k\}$, then $f(x_k)\to f(\bar{x})$.
  \item[(3)] If $f$ is continuously differentiable and satisfies Assumption \ref{assump:bound}, and if the sequence $\{t_k\}$ is bounded below by a positive constant $t_{\min}$,
  then any accumulation point of the sequence $\{x_k\}$ is a stationary point.
\end{description}
\end{thm}
\begin{proof}
We consider only the first case of Theorem \ref{thm:IRPG-decrease}.
The other case can be handled similarly.

\noindent(1) From \eqref{eq:decrease-1}, we obtain
\begin{equation}\label{}
\frac{1}{4t_k} \|\eta_{x_k}^*\|^2 \leq F(x_k) - F(x_{k+1}) + \frac{2-q}{2\rho^{\frac{q}{2-q}}} \delta_k^{\frac{2}{2-q}}.
\end{equation}
Summing both sides of this inequality from $k=0$ to $n-1$, we obtain
\begin{equation*}
  \sum_{k=0}^{n-1} \frac{1}{4t_k} \|\eta_{x_k}^*\|^2 \leq F(x_0) - F(x_{n}) + \frac{2-q}{2\rho^{\frac{q}{2-q}}} \sum_{k=0}^{n-1} \delta_k^{\frac{2}{2-q}}.
\end{equation*}
Letting $n\to\infty$ and using the fact that $F(x_{n})\ge F^*$, we obtain
\begin{equation}\label{eq:convergence-1}
\sum_{k=0}^{\infty} \frac{1}{ t_k} \|\eta_{x_k}^*\|^2 \leq 4(F(x_0) - F^*) +  \frac{2(2-q)}{\rho^{\frac{q}{2 - q}}} \sum_{k=0}^{\infty} \delta_k^{\frac{2}{2 - q}}<\infty.
\end{equation}
Thus, $\lim_{k\rightarrow\infty}\frac{1}{ t_k} \|\eta_{x_k}^*\|^2=0$.
Since $t_k \leq \frac{1}{2(L_k + q\rho)}\le \frac{1}{2q\rho}$, it follows that
$\lim_{k\rightarrow\infty}\|\eta_{x_k}^*\|=0$.

\noindent(2) Defining $A_k:=\frac{2-q}{2\rho^{\frac{q}{2-q}}}\sum_{i=k}^\infty \delta_i^{\frac{2}{2-q}}$,
we have $A_k\to 0$ as $k\to \infty$ and $A_k-A_{k+1}=\frac{2-q}{2\rho^{\frac{q}{2-q}}} \delta_k^{\frac{2}{2-q}}$
for all $k\ge0$. Then inequality \eqref{eq:decrease-1} can be rewritten as
\begin{equation}\label{eq:approximate descent}
F(x_{k+1})+A_{k+1} \leq F(x_k) + A_k - \frac{1}{4t_k} \|\eta_{x_k}^*\|^2.
\end{equation}
Thus, the sequence $\{F(x_k) + A_k\}$ is nonincreasing.
Since $F(x_k)\ge F^*$ and $A_k\to 0$, the sequence $\{F(x_k) + A_k\}$ is bounded from below, and hence convergent.
Given that $A_k\to 0$, the sequence $\{F(x_k)\}$ is convergent as well.
Since $\bar{x}$ is an accumulation point of the sequence $\{x_k\}$, there exists a subsequence $\{x_{k_j}\}$ such that $x_{k_j} \to \bar{x}$ as $j \to \infty$.
We conclude that $\lim_{k\to\infty}F(x_k)=\lim_{j\to\infty}F(x_{k_j})=F(\bar{x})$.

\noindent(3) Let $\bar{x}$ be an accumulation point of the sequence $\{x_k\}$, and $\{x_{k_j}\}$ be a subsequence such that $x_{k_j} \to \bar{x}$ as $j \to \infty$.

Let $g_k=g_{\delta_k, L_k, q}(x_k)$. By the optimality condition of the subproblem in Algorithm \ref{alg:algorithm-IRPG},
there holds
$$
0 \in  g_{k_j} + \frac{1}{t_{k_j}} \eta_{x_{k_j}}^* + [ D R_{x_{k_j}}(\eta_{x_{k_j}}^*) ]^\sharp \partial h(x_{k_j+1}).
$$
Since
$\lim_{j \to \infty} [ D R_{x_{k_j}}(\eta_{x_{k_j}}^*) ]^\sharp = [ D R_{\bar{x}}(0) ]^\sharp = \mathrm{id}_{T_{\bar{x}}\mathcal{M}}$,
$[ D R_{x_{k_j}}(\eta_{x_{k_j}}^*) ]^\sharp$ is invertible for sufficiently large $j$. Hence
$$
- [ D R_{x_{k_j}}(\eta_{x_{k_j}}^*) ]^{-\sharp} \left( g_{k_j} + \frac{1}{t_{k_j}} \eta_{x_{k_j}}^* \right) \in \partial h(x_{k_j+1}).
$$

By Assumption \ref{assump:bound}, we have $\| g_{k_j} - \grad f(x_{k_j}) \| \to 0$.
Furthermore, since $\grad f(x_{k_j}) \to \grad f(\bar{x})$, we have $g_{k_j} \to \grad f(\bar{x})$
as $j \to \infty$.
Since $t_{k_j}\ge t_{\min}$, we have $\|\frac{1}{t_{k_j}} \eta_{x_{k_j}}^*\|\le \frac{1}{t_{\min}}\|\eta_{x_{k_j}}^*\|$,
and hence $\frac{1}{t_{k_j}} \eta_{x_{k_j}}^*\to 0$ as $j \to \infty$.
Therefore, we obtain
\begin{equation*}
  - [ D R_{x_{k_j}}(\eta_{x_{k_j}}^*) ]^{-\sharp} \left( g_{k_j} + \frac{1}{t_{k_j}} \eta_{x_{k_j}}^* \right)
  \to - \mathrm{grad} f(\bar{x}) \ \
  \text{as}\ \ j\to\infty.
\end{equation*}
We also have $x_{k_j+1} = R_{x_{k_j}}(\eta_{x_{k_j}}^*) \to R_{\bar{x}}(0) = \bar{x}$ as $j \to \infty$.
Moreover, since $h$ is continuous, $h(x_{k_j+1}) \to h(\bar{x})$.
It follows from Lemma \ref{lem-Hosseini} that
$- \mathrm{grad} f(\bar{x}) \in \partial h(\bar{x})$,
that is, $0 \in \mathrm{grad} f(\bar{x}) + \partial h(\bar{x}) = \partial F(\bar{x})$,
and hence $\bar{x}$ is a stationary point.
\end{proof}

\subsection{Global sequential convergence under the Riemannian Kurdyka--{\L}ojasiewicz property}

The Kurdyka--{\L}ojasiewicz (KL) property has been widely used in the convergence analysis of various convex
and nonconvex optimization algorithms in the Euclidean setting \cite{Bolte-2007,Bolte-2009,Bolte-2013,Bolte-2014,Li-2023}.
In what follows, we employ the Riemannian KL property to study the global sequential convergence of Algorithm \ref{alg:algorithm-IRPG},
that is, the full sequence of iterates converges to a single limit point, and this limit point is a stationary point.

\begin{de}\label{de:Riemannian-KL}
\cite{Huang-2022,Li-2023} A function $f: \mathcal{M} \to \mathbb{R}$ is said to satisfy the Riemannian KL property
at $\bar{x} \in \mathcal{M}$ if there exist $\varepsilon \in (0, \infty]$,
a neighborhood $U \subset \mathcal{M}$ of $\bar{x}$, and a continuous concave function $\varrho: [0, \varepsilon) \to [0, \infty)$ with
\begin{equation*}
  \varrho(0) = 0,~~~~\varrho\in C^1(0,\varepsilon),~~~~\text{and}~~~~\varrho'(x) > 0~~~\forall\,x\in(0,\varepsilon)
\end{equation*}
such that for all $x\in U\cap \{x\in \mathcal{M}:0<|f(x)-f(\bar{x})|<\varepsilon\}$ the KL inequality holds, i.e.,
\begin{equation*}
  \varrho'(|f(x) - f(\bar{x})|) \operatorname{dist}(0, \partial f(x)) \geq 1,
\end{equation*}
where $\mathrm{dist}(0, \partial f(x)) := \inf\{\|v\|_x: v \in \partial f(x)\}$.
The function $\varrho$ is referred to as the desingularizing function.
\end{de}

In \cite{Li-2023}, Li et al. established a KL inequality in the Euclidean setting, which serves as a slightly stronger variant of the classical formulation.
Specifically, the authors added an absolute value to $f(x) - f(\bar{x})$ in the original KL inequality.
However, this modified version of KL inequality also holds for semialgebraic and subanalytic functions, which underscores its generality.
Therefore, we adopt this modified version in the present work to define the corresponding Riemannian KL inequality.

The following assumption regarding the desingularizing function was first introduced in \cite{Li-2023} and recently adopted in \cite{Deng-2025}.

\begin{assump}\label{assump:desingularizing-function}
\cite{Li-2023,Deng-2025} The desingularizing function $\varrho$ satisfies the quasi-additivity-type property, that is, there exists a constant $C_\varrho>0$ such that
\begin{equation}\label{}
[\varrho'(x+y)]^{-1} \leq C_\varrho\left[(\varrho'(x))^{-1} + (\varrho'(y))^{-1}\right], \quad \forall\, x,y \in (0,\varepsilon) \text{ with } x+y < \varepsilon.
\end{equation}
\end{assump}

The following lemma shows that if the Riemannian KL property holds at every point in a compact set on which the function is constant, then there exists
a common desingularizing function such that the Riemannian KL property holds for all
points in the compact set.

\begin{lem}\label{lem:Riemannian-KL}
\cite{Huang-2022} Let $\bar{\Omega}$ be a compact set in $\mathcal{M}$ and let $f: \mathcal{M} \to (-\infty, \infty]$ be a continuous function. Assume that $f$ is constant on $\bar{\Omega}$ and satisfies the Riemannian KL property at every point of $\bar{\Omega}$. Then, there exist $\varpi > 0$, $\varepsilon > 0$, and a continuous concave function $\varrho: [0, \varepsilon) \to [0, \infty)$ such that for all $\bar{x}$ in $\bar{\Omega}$ and all $x$ in the following intersection:
\begin{equation*}
\{x \in \mathcal{M}: \operatorname{dist}(x, \bar{\Omega}) < \varpi\} \cap \{x \in \mathcal{M}: 0 < |f(x)-f(\bar{x})| < \varepsilon\},
\end{equation*}
one has
\begin{equation*}
\varrho'(|f(x) - f(\bar{x})|) \operatorname{dist}(0, \partial f(u)) \geq 1.
\end{equation*}
\end{lem}

In order to study the global sequential convergence of Algorithm \ref{alg:algorithm-IRPG} based on the Riemannian
KL property, we also require two results regarding the retraction and vector transport,
given in Lemmas \ref{lem:Huang-Retr-estimate} and \ref{lem:Huang-estimate}, respectively.

\begin{lem}\label{lem:Huang-Retr-estimate}
\cite{Huang-2022}
Let $\bar{\Omega}\subset \mathcal{M}$ be a compact set. Then, for any given $\delta_T>0$,
there exists a constant $\kappa_0>0$ such that
\begin{equation}\label{}
  \mathrm{dist}(x,R_x(\eta))\le \kappa_0\|\eta\|
\end{equation}
for all $x\in \bar{\Omega}$ and $\eta \in T_x\mathcal{M}$ satisfying $\|\eta\| \le \delta_T$.
\end{lem}

\begin{lem}\label{lem:Huang-estimate}
\cite{Huang-2022}
Let $\xi$ be a locally Lipschitz continuous vector field on $\mathcal{M}$. Given a constant $a\in\br$ and a compact set $\bar{\Omega} \subset \mathcal{M}$, there exist positive constants $\mu$, $L_u$, and $L_v$ such that
$$
\|\xi_{y} - \mathcal{T}_{\eta_x}^{-\sharp} (\xi_x + a\eta_x)\| \leq (L_u+|a|L_v) \|\eta_x\|
$$
for all $x\in \bar{\Omega}$ and $\eta_x \in T_x\mathcal{M}$ satisfying $\|\eta_x\| \le \mu$, where $y = R_x(\eta_x)$.
\end{lem}

To apply Lemmas \ref{lem:Huang-Retr-estimate} and \ref{lem:Huang-estimate} in the subsequent proof, we impose the following assumption.

\begin{assump}\label{assump:compact}
%Let $\{x_k\}$ be a sequence in $\mathcal{M}$.
There exists a compact set $\bar{\Omega} \subset \mathcal{M}$
such that the sequence $\{x_k\}$ lies entirely within $\bar{\Omega}$.
\end{assump}

Now, we are in a position to establish the convergence of the iterates $\{x_k\}$ generated by
Algorithm \ref{alg:algorithm-IRPG} to a single stationary point.
The proof draws on arguments from \cite[Theorem 4]{Huang-2022}, \cite[Theorem 2]{Deng-2025}, and \cite[Theorem 3.6]{Li-2023}.

\begin{thm}\label{thm:point-convergence}
Suppose that the conditions of Theorem \ref{thm:IRPG-convergence} hold,
that the sequence of oracle errors $\{\delta_k\}$ is positive,
and that $f$ is locally Lipschitz continuously differentiable on $\mathcal{M}$.
Let the sequence $\{x_k\}$ be generated by Algorithm \ref{alg:algorithm-IRPG} and satisfy Assumption \ref{assump:compact},
and let $\mathcal{S}$ denote the set of all accumulation points of $\{x_k\}$.
Suppose that $F$ satisfies the Riemannian KL property at every point in $\mathcal{S}$ with
the desingularizing function $\varrho$ satisfying Assumption \ref{assump:desingularizing-function}.
Suppose further that
\begin{equation}\label{eq:thm-KL-condition}
  \sum_{k=0}^\infty\left(\varrho'\big(\frac{2-q}{2\rho^{\frac{q}{2-q}}}\sum_{i=k}^\infty \delta_i^{\frac{2}{2-q}}\big)\right)^{-1}< \infty.
\end{equation}
Then, for any $\bar{x}\in \mathcal{S}$, we have $x_k\to \bar{x}$ as $k\to\infty$. It follows that $\mathcal{S}$ contains only a single point.
\end{thm}
\begin{proof}
If there exists $l_0\geq 0$ such that $\eta_{x_{l_0}}^*=0$,
then $x_k\equiv x_{l_0}$ for all $k\ge l_0$, in which case the theorem holds trivially.
Therefore, in what follows we assume that $\eta_{x_k}^*\neq 0$ for all $k\ge0$.

Note that the global convergence result in Theorem \ref{thm:IRPG-convergence} implies that
every point in $\mathcal{S}$ is a stationary point.
Since $\lim_{k\to\infty}\|\eta_{x_k}^*\|=0$,
there exists $\delta_T>0$ such that $\|\eta_{x_k}^*\|\le\delta_T$ for all $k\ge0$.
Thus, the application of Lemma \ref{lem:Huang-Retr-estimate} implies that
\begin{equation}\label{eq:dist}
\mathrm{dist}(x_k, x_{k+1}) = \mathrm{dist}(x_k, R_{x_k}(\eta_{x_k}^*)) \leq \kappa_0 \|\eta_{x_k}^*\| \to 0.
\end{equation}
Then, by \cite[Remark 5]{Bolte-2014}, it follows that $\mathcal{S}$ is a compact set.
Moreover, since the sequence $\{F(x_k)\}$ is convergent, $F$ takes
the same value at every point of $\mathcal{S}$. Therefore, by
Lemma \ref{lem:Riemannian-KL}, there exists a common desingularizing function,
denoted $\varrho$, for the Riemannian KL property of $F$ to hold
at every point of $\mathcal{S}$.

In what follows, without loss of generality, we assume that $F(x_k)\neq F(\bar{x})$ for all $k\ge0$.
Since $F(x_k)\to F(\bar{x})$ and $\operatorname{dist}(x_k, \mathcal{S})\to 0$, by the
Riemannian KL property of $F$ on $\mathcal{S}$, there exists
$l_1>0$ such that
\begin{equation}\label{eq:KL-estimate}
\varrho'(|F(x_k)-F(\bar{x})|) \operatorname{dist}(0, \partial F(x_k)) \geq 1 \quad \text{for all } k\geq l_1.
\end{equation}

Let $g_k=g_{\delta_k, L_k, q}(x_k)$.
By Assumption \ref{assump:bound}, there exists a constant $C_q\ge0$ such that
\begin{equation}\label{}
\|g_k-\grad f(x_k)\|\le C_q\, \delta_k^{\frac{2}{2-q}}.
\end{equation}

Since $\lim_{k\to\infty}\|\eta_{x_k}^*\|=0$,
there exists $l_2>0$ such that $\|\eta_{x_k}^*\|\le \mu$ for all $k\ge l_2$, where
$\mu>0$ is the constant from Lemma \ref{lem:Huang-estimate}.
Note also that $\frac{1}{t_k} \leq \frac{1}{t_{\min}}$.
Then, by Lemma \ref{lem:Huang-estimate}, there exists a constant $L_c>0$ such that
\begin{equation*}
  \big\|\grad f(x_{k+1})- \mathcal{T}_{\eta_{x_{k}}^*}^{-\sharp} \big( \grad f(x_{k}) + \frac{1}{t_{k}} \eta_{x_{k}}^* \big)\big\|
  \le L_c\big\|\eta_{x_{k}}^*\big\|.
\end{equation*}

Since $\mathcal{T}_{\eta}^{-\sharp}$ is smooth with respect to $\eta$
and the set $\{\eta\in T_x\mathcal{M}:x\in\bar{\Omega},~\|\eta\|\le \mu\}$ is compact,
there exists a constant $L_t>0$ such that
\begin{equation*}
  \big\|\mathcal{T}_{\eta_{x_{k}}^*}^{-\sharp}\big\|\le L_t \quad \text{for all} \ k\ge l_2.
\end{equation*}

Therefore, we have
\begin{equation*}
\begin{aligned}
&~~~~\big\|\grad f(x_{k+1})- \mathcal{T}_{\eta_{x_{k}}^*}^{-\sharp} \big( g_k + \frac{1}{t_{k}} \eta_{x_{k}}^* \big)\big\| \\
&\le \big\|\grad f(x_{k+1})- \mathcal{T}_{\eta_{x_{k}}^*}^{-\sharp} \big( \grad f(x_{k}) + \frac{1}{t_{k}} \eta_{x_{k}}^* \big)\big\|
+ \big\|\mathcal{T}_{\eta_{x_{k}}^*}^{-\sharp}\big\|\cdot\big\|g_k-\grad f(x_k)\big\| \\
&\le L_c\big\|\eta_{x_{k}}^*\big\| + L_tC_q \delta_k^{\frac{2}{2-q}}.
\end{aligned}
\end{equation*}

From the optimality condition for the subproblem in Algorithm \ref{alg:algorithm-IRPG}, it follows that
$- \mathcal{T}_{\eta_{x_{k}}^*}^{-\sharp} \big( g_k + \frac{1}{t_{k}} \eta_{x_{k}}^* \big)\in \partial h(x_{k+1})$.
Thus
\begin{equation*}
  \grad f(x_{k+1})- \mathcal{T}_{\eta_{x_{k}}^*}^{-\sharp} \big( g_k + \frac{1}{t_{k}} \eta_{x_{k}}^* \big) \in \partial F(x_{k+1}).
\end{equation*}
Therefore, we obtain
\begin{equation}\label{}
  \operatorname{dist}(0, \partial F(x_k))\le L_c\|\eta_{x_{k-1}}^*\| + L_tC_q \delta_{k-1}^{\frac{2}{2-q}}.
\end{equation}
Inserting this inequality into \eqref{eq:KL-estimate} yields
\begin{equation}\label{eq:KL-estimate-2}
  \varrho'(|F(x_k)-F(\bar{x})|)\ge \frac{1}{L_c\|\eta_{x_{k-1}}^*\| + L_tC_q \delta_{k-1}^{\frac{2}{2-q}}}
  \quad \text{for all}\ k\ge k_0:=\max\{l_1,l_2\}.
\end{equation}

We now consider the first case in Theorem \ref{thm:IRPG-decrease}. The other case can be handled similarly.
Define $A_k:=\frac{2-q}{2\rho^{\frac{q}{2-q}}}\sum_{i=k}^\infty \delta_i^{\frac{2}{2-q}}$.
Then \eqref{eq:decrease-1} can be rewritten as
\begin{equation}\label{eq:decrease-rewritten}
F(x_{k+1})+A_{k+1} \leq F(x_k) + A_k - \frac{1}{4t_k} \|\eta_{x_k}^*\|^2,
\end{equation}
which, together with the fact that $\eta_{x_k}^*\neq 0$,
implies that the sequence $\{F(x_k) + A_k\}$ is strictly decreasing and converges to $F(\bar{x})$.
Hence $F(x_k)-F(\bar{x})+A_k>0$ for all $k\ge0$.
Define $u_k:=\varrho(F(x_k)-F(\bar{x})+A_k)$.
Then
\begin{equation*}
\begin{aligned}
u_k-u_{k+1} &\ge \varrho'(F(x_k)-F(\bar{x})+A_k)(F(x_k)+A_k - F(x_{k+1})-A_{k+1})  \\
&\ge \varrho'(|F(x_k)-F(\bar{x})|+A_k)(F(x_k)+A_k - F(x_{k+1})-A_{k+1}) \\
&\ge \frac{1}{C_\varrho}\cdot \frac{1}{[\varrho'(|F(x_k)-F(\bar{x})|)]^{-1}+[\varrho'(A_k)]^{-1}}\cdot \frac{1}{4t_k} \|\eta_{x_k}^*\|^2 \\
&\ge \frac{1}{C_\varrho}\cdot \frac{1}{L_c\|\eta_{x_{k-1}}^*\| + L_tC_q \delta_{k-1}^{\frac{2}{2-q}}
+[\varrho'(A_k)]^{-1}}\cdot \frac{q\rho}{2} \|\eta_{x_k}^*\|^2,
\end{aligned}
\end{equation*}
where the first inequality follows from the concavity of $\varrho$,
the second inequality follows from
the fact that $\varrho'$ is monotonically decreasing (since $\varrho$ is concave),
the third inequality follows from \eqref{eq:decrease-rewritten} and Assumption \ref{assump:desingularizing-function},
and the last inequality follows from \eqref{eq:KL-estimate-2} and the fact that $\frac{1}{t_k}\ge 2(L_k+q\rho)\ge 2q\rho$.

Let $M:=\frac{q\rho}{2C_\varrho\max\{L_c,L_tC_q,1\}}$. We have
\begin{equation}\label{}
  u_k-u_{k+1} \ge \frac{M\|\eta_{x_k}^*\|^2}{\|\eta_{x_{k-1}}^*\| +  \delta_{k-1}^{\frac{2}{2-q}}
+[\varrho'(A_k)]^{-1}}.
\end{equation}
Then
\begin{equation*}
  \|\eta_{x_k}^*\|^2\le \frac{1}{M}(u_k-u_{k+1})\big(\|\eta_{x_{k-1}}^*\| +  \delta_{k-1}^{\frac{2}{2-q}}
+[\varrho'(A_k)]^{-1}\big).
\end{equation*}
Taking the square root of both sides of the above inequality and noting $2\sqrt{ab}\le a+b$ for all $a,b\ge0$, we obtain
\begin{equation*}
  2\|\eta_{x_k}^*\| \le \frac{1}{M}(u_k-u_{k+1}) + \|\eta_{x_{k-1}}^*\| +  \delta_{k-1}^{\frac{2}{2-q}}
+[\varrho'(A_k)]^{-1}.
\end{equation*}
Summing both sides of the above inequality over $k=k_0,k_0+1,\cdots$ yields
\begin{equation*}
\begin{aligned}
2\sum_{k=k_0}^\infty\|\eta_{x_k}^*\| &\le \frac{1}{M}\sum_{k=k_0}^\infty(u_k-u_{k+1}) + \sum_{k=k_0}^\infty\|\eta_{x_{k-1}}^*\| +  \sum_{k=k_0}^\infty\delta_{k-1}^{\frac{2}{2-q}}
+\sum_{k=k_0}^\infty[\varrho'(A_k)]^{-1} \\
&\le \frac{1}{M}u_{k_0} + \|\eta_{x_{k_0-1}}^*\| + \sum_{k=k_0}^\infty\|\eta_{x_k}^*\|+  \sum_{k=k_0}^\infty\delta_{k-1}^{\frac{2}{2-q}}
+\sum_{k=k_0}^\infty[\varrho'(A_k)]^{-1},
\end{aligned}
\end{equation*}
which implies that
\begin{equation}\label{eq:recur}
  \sum_{k=k_0}^\infty\|\eta_{x_k}^*\| \le \frac{1}{M}u_{k_0} + \|\eta_{x_{k_0-1}}^*\| +  \sum_{k=k_0}^\infty\delta_{k-1}^{\frac{2}{2-q}}
+\sum_{k=k_0}^\infty[\varrho'(A_k)]^{-1} < \infty.
\end{equation}
Therefore, $\sum_{k=0}^\infty\|\eta_{x_k}^*\|< \infty$.
Inserting this into \eqref{eq:dist} yields
\begin{equation*}
  \sum_{k=0}^\infty \mathrm{dist}(x_k,x_{k+1}) \le \kappa_0 \sum_{k=0}^\infty\|\eta_{x_k}^*\|< \infty,
\end{equation*}
which implies that $\{x_k\}$ is a Cauchy sequence and hence converges.
The proof is complete.
\end{proof}

\noindent
\textbf{Remark:} We give a specific choice of $\delta_k$ that satisfies condition \eqref{eq:thm-KL-condition}.
  Suppose the desingularizing function is a {\L}ojasiewicz function, i.e., $\varrho(x)=cx^{1-\theta}$,
  where $c>0$ and $\theta\in (0,1)$.
  Then $(\varrho'(x))^{-1}=\frac{1}{c(1-\theta)}x^{\theta}$.
  We set $\delta_k=\frac{1}{(k+2)^p}$, where $p>\frac{2-q}{2}(1+1/\theta)$.
  Let $r:=\frac{2p}{2-q}$. Then $r>1+1/\theta$. For all $k\ge0$, we have
  \begin{equation*}
    \sum_{i=k}^\infty \delta_i^{\frac{2}{2-q}} = \sum_{i=k}^\infty\frac{1}{(i+2)^r}
    \le \int_{k+1}^\infty \frac{1}{x^r}dx = \frac{1}{(r-1)(k+1)^{r-1}}.
  \end{equation*}
  Therefore, we obtain the stated condition
  \begin{equation*}
    \sum_{k=0}^\infty\left(\varrho'\big(\frac{2-q}{2\rho^{\frac{q}{2-q}}}\sum_{i=k}^\infty \delta_i^{\frac{2}{2-q}}\big)\right)^{-1}
    \le \frac{1}{c(1-\theta)}\left(\frac{2-q}{2(r-1)\rho^{\frac{q}{2-q}}}\right)^{\theta}\sum_{k=0}^\infty\frac{1}{(k+1)^{(r-1)\theta}}
    < \infty.
  \end{equation*}

\subsection{Convergence rate under the Riemannian Kurdyka--{\L}ojasiewicz property}

To establish the convergence rate under the Riemannian KL property, we first prove a lemma
which states that a recursive inequality between $D_k$ and $D_{k-1}$ yields the decay rate of $D_k$.

\begin{lem}\label{lem:rate}
Let $\{D_k\}$ be a monotonically decreasing sequence converging to $0$.
If there exists $k_0\in\bn$ such that for all $k\ge k_0$ there holds
\begin{equation}\label{eq:recurrence}
  D_k \le (D_{k-1}-D_k) + C_1(D_{k-1}-D_k)^\alpha + \frac{C_2}{(k+1)^\beta},
\end{equation}
where $\alpha,\beta,C_1,C_2$ are positive constants, then
\begin{equation*}
D_k =
\begin{cases}
\mathcal{O}(k^{-\beta}) & \text{if }\ \alpha\ge 1, \\
\mathcal{O}(k^{-\min\{\frac{\alpha}{1-\alpha},\,\beta\}}) & \text{if }\ 0<\alpha<1.
\end{cases}
\end{equation*}
\end{lem}
\begin{proof}
We prove this lemma by considering four distinct cases.

Case 1: $\alpha=1$.

Inequality \eqref{eq:recurrence} can be rewritten as
\begin{equation}\label{eq:recurrence-case-1}
  D_k \le \gamma D_{k-1} + \frac{C'}{(k+1)^\beta},
\end{equation}
where $\gamma:=\frac{1+C_1}{2+C_1}\in (0,1)$ and $C'=\frac{C_2}{2+C_1}>0$.

Since $1-\gamma(\frac{k+1}{k})^\beta \to 1-\gamma$ as $k\to\infty$,
there exists $k_1\ge k_0$ such that for all $k\ge k_1$, there holds $1-\gamma(\frac{k+1}{k})^\beta\ge\frac{1-\gamma}{2}$.
We define
\begin{equation}\label{}
  M:=\max\left\{\frac{2C'}{1-\gamma},(k_1+1)^\beta D_{k_1}\right\}.
\end{equation}
We now prove by induction that $D_k\le \frac{M}{(k+1)^\beta}$ for all $k\ge k_1$.
When $k=k_1$, the conclusion follows immediately from the definition of $M$.
We assume that $D_{k-1}\le \frac{M}{k^\beta}$.
Since $M\ge \frac{2C'}{1-\gamma}$ and $1-\gamma(\frac{k+1}{k})^\beta\ge\frac{1-\gamma}{2}$,
we have $M[1-\gamma(\frac{k+1}{k})^\beta]\ge C'$,
which is equivalent to $\gamma\frac{M}{k^\beta}+\frac{C'}{(k+1)^\beta}\le \frac{M}{(k+1)^\beta}$.
Combining this inequality with \eqref{eq:recurrence-case-1} and the induction hypothesis $D_{k-1}\le \frac{M}{k^\beta}$, we obtain
\begin{equation*}
  D_k \le \gamma D_{k-1} + \frac{C'}{(k+1)^\beta}\le \gamma\frac{M}{k^\beta}+\frac{C'}{(k+1)^\beta}\le \frac{M}{(k+1)^\beta}.
\end{equation*}

Case 2: $\alpha>1$.

Since $(D_{k-1}-D_k)^\alpha/(D_{k-1}-D_k)=(D_{k-1}-D_k)^{\alpha-1}\to 0$ as $k\to \infty$,
the sequence $\{(D_{k-1}-D_k)^\alpha/(D_{k-1}-D_k)\}$ is bounded,
that is, there exists $M_1>0$ such that
$(D_{k-1}-D_k)^\alpha\le M_1 (D_{k-1}-D_k)$.
Substituting this inequality into \eqref{eq:recurrence}, we reduce the problem to Case 1.

Case 3: $0<\alpha<1$ and $\tau:=\frac{\alpha}{1-\alpha}\le \beta$.

Using \eqref{eq:recurrence} and the fact that $\tau\le \beta$, we obtain
\begin{equation*}
  D_k \le (D_{k-1}-D_k) + C_1(D_{k-1}-D_k)^\alpha + \frac{C_2}{(k+1)^\tau}.
\end{equation*}
Since $(D_{k-1}-D_k)/(D_{k-1}-D_k)^\alpha=(D_{k-1}-D_k)^{1-\alpha}\to 0$ as $k\to \infty$,
the sequence $\{(D_{k-1}-D_k)/(D_{k-1}-D_k)^\alpha\}$ is bounded,
that is, there exists $M_2>0$ such that
$(D_{k-1}-D_k)\le M_2 (D_{k-1}-D_k)^\alpha$. Thus we have
\begin{equation}\label{eq:case-3-estimate}
  D_k - \frac{C_2}{(k+1)^\tau} \le  C_3(D_{k-1}-D_k)^\alpha,~~\text{where}~~C_3:=M_2+C_1.
\end{equation}

We take $M>0$ sufficiently large such that
\begin{equation}\label{}
  M\ge (k_0+1)^\tau D_{k_0},~~M\ge C_2,~~(M-C_2)^{1/\alpha}\ge 2C_3^{1/\alpha} C_\tau M,
\end{equation}
where $C_\tau:=\tau\cdot\max\{1,2^{\tau-1}\}$.
We now prove by induction that $D_k\le \frac{M}{(k+1)^\tau}$ for all $k\ge k_0$.
When $k=k_0$, the conclusion follows immediately from the definition of $M$.
We assume that $D_{k-1}\le \frac{M}{k^\tau}$.
We now prove that $D_k\le \frac{M}{(k+1)^\tau}$.

If $D_k \le \frac{C_2}{(k+1)^\tau}$, the conclusion follows immediately from the definition of $M$.

If $D_k > \frac{C_2}{(k+1)^\tau}$, from \eqref{eq:case-3-estimate} and the induction hypothesis $D_{k-1}\le \frac{M}{k^\tau}$, we obtain
\begin{equation*}
  \frac{1}{C_3^{1/\alpha}}\left(D_k - \frac{C_2}{(k+1)^\tau}\right)^{1/\alpha}
  \le D_{k-1}-D_k \le \frac{M}{k^\tau} - D_k.
\end{equation*}
Multiplying both sides of this inequality by $(k+1)^\tau$, we obtain
\begin{equation}\label{eq:case-3-estimate2}
  \frac{1}{C_3^{1/\alpha}}(D_k(k+1)^\tau - C_2)^{1/\alpha}(k+1)^{\tau-\frac{\tau}{\alpha}}
  \le M\left(1+\frac{1}{k}\right)^\tau - D_k(k+1)^\tau.
\end{equation}
Since $\tau-\frac{\tau}{\alpha}=-1$, the above inequality can be rewritten as
\begin{equation*}
  D_k(k+1)^\tau + \frac{(D_k(k+1)^\tau - C_2)^{1/\alpha}}{C_3^{1/\alpha}(k+1)} \le M\left(1+\frac{1}{k}\right)^\tau.
\end{equation*}
By the mean value theorem, we readily verify that $\left(1+\frac{1}{k}\right)^\tau\le 1+\frac{C_\tau}{k}$,
where $C_\tau=\tau\cdot\max\{1,2^{\tau-1}\}$.
We proceed by contradiction and suppose that $D_k > \frac{M}{(k+1)^\tau}$. It follows that
\begin{equation*}
  M + \frac{(M - C_2)^{1/\alpha}}{C_3^{1/\alpha}(k+1)}< M + \frac{MC_\tau}{k},
\end{equation*}
which implies
\begin{equation*}
  (M - C_2)^{1/\alpha} < C_3^{1/\alpha} C_\tau M\frac{k+1}{k}\le 2C_3^{1/\alpha} C_\tau M,
\end{equation*}
which contradicts the definition of $M$.

Case 4: $0<\alpha<1$ and $\tau:=\frac{\alpha}{1-\alpha}> \beta$.

The proof is analogous to that of Case 3.
We only highlight the key differences.
By a similar derivation, the inequality corresponding to \eqref{eq:case-3-estimate} becomes
\begin{equation}\label{}
  D_k - \frac{C_2}{(k+1)^\beta} \le  C_3(D_{k-1}-D_k)^\alpha.
\end{equation}
We then take $M>0$ sufficiently large such that
\begin{equation}\label{}
  M\ge (k_0+1)^\beta D_{k_0},~~M\ge C_2,~~(M-C_2)^{1/\alpha}\ge 2C_3^{1/\alpha} C_\beta M,
\end{equation}
where $C_\beta:=\beta\cdot\max\{1,2^{\beta-1}\}$.
We now prove by induction that $D_k\le \frac{M}{(k+1)^\beta}$ for all $k\ge k_0$.
When $k=k_0$, the conclusion follows immediately from the definition of $M$.
We assume that $D_{k-1}\le \frac{M}{k^\beta}$.
We now prove that $D_k\le \frac{M}{(k+1)^\beta}$.
If $D_k \le \frac{C_2}{(k+1)^\beta}$, the conclusion follows immediately from the definition of $M$.
If $D_k > \frac{C_2}{(k+1)^\beta}$, we derive an inequality analogous to \eqref{eq:case-3-estimate2} of Case 3 as follows:
\begin{equation}\label{}
  \frac{1}{C_3^{1/\alpha}}(D_k(k+1)^\beta - C_2)^{1/\alpha}(k+1)^{\beta-\frac{\beta}{\alpha}}
  \le M\left(1+\frac{1}{k}\right)^\beta - D_k(k+1)^\beta.
\end{equation}
Let $\sigma := \frac{\beta}{\alpha}-\beta=\frac{\beta}{\tau}\in(0,1)$. Then we have
\begin{equation*}
  D_k(k+1)^\beta + \frac{(D_k(k+1)^\beta - C_2)^{1/\alpha}}{C_3^{1/\alpha}(k+1)^\sigma} \le M\left(1+\frac{1}{k}\right)^\beta.
\end{equation*}
The rest of the proof follows the same derivations as in Case 3.
\end{proof}

The following theorem shows that if $F$ satisfies the Riemannian KL property with
a desingularizing function of the form $\varrho(x)=cx^{1-\theta}$,
where $c>0$ and $\theta\in (0,1)$, then an explicit convergence rate can be derived.

\begin{thm}\label{}
Suppose that the conditions of Theorem \ref{thm:point-convergence} hold.
Assume in addition that the desingularizing function has the form $\varrho(x)=cx^{1-\theta}$,
where $c>0$ and $\theta\in (0,1)$,
and the oracle error has the form $\delta_k=\frac{1}{(k+2)^p}$, where $p>\frac{2-q}{2}(1+1/\theta)$.
Then
\begin{equation*}
\mathrm{dist}(x_k,\bar{x}) =
\begin{cases}
\mathcal{O}(k^{-[(r-1)\theta-1]}) & \text{if}\ 0<\theta\le \frac{1}{2},  \\
\mathcal{O}(k^{-[(r-1)\theta-1]}) & \text{if}\ \frac{1}{2}<\theta<1\ \text{and}\ \frac{\theta+1}{\theta}<r\le \frac{2\theta}{2\theta-1}, \\
\mathcal{O}(k^{-\frac{1-\theta}{2\theta-1}}) & \text{if}\ \frac{1}{2}<\theta<1\ \text{and}\ r> \frac{2\theta}{2\theta-1},
\end{cases}
\end{equation*}
where $\bar{x}$ denotes the unique accumulation point of the sequence $\{x_k\}$, and $r := \frac{2p}{2-q}$.
\end{thm}
\begin{proof}
Let $D_k:=\sum_{i=k}^\infty\|\eta_{x_i}^*\|$.
We first prove that the following recursive inequality between $D_k$ and $D_{k-1}$ holds for sufficiently large $k\in\bn$:
\begin{equation}\label{eq:true-recurrence}
  D_k \le (D_{k-1}-D_k) + C_1(D_{k-1}-D_k)^{\frac{1-\theta}{\theta}} + \frac{C_2}{(k+1)^{\beta}},
\end{equation}
where $\beta:=\min\big\{(r-1)(1-\theta),(r-1)\theta-1\big\}$,
$C_1$ and $C_2$ are positive constants.

It follows from \eqref{eq:recur} that for all $k\ge k_0$, we have
\begin{equation}\label{eq:recur-rewritten}
  \sum_{i=k}^\infty\|\eta_{x_i}^*\| \le \frac{1}{M}u_{k} + \|\eta_{x_{k-1}}^*\| +  \sum_{i=k}^\infty\delta_{i-1}^{\frac{2}{2-q}}
+\sum_{i=k}^\infty[\varrho'(A_i)]^{-1},
\end{equation}
where $u_k=c(F(x_k)-F(\bar{x})+A_k)^{1-\theta}$ and $A_k = \frac{2-q}{2\rho^{\frac{q}{2-q}}}\sum_{i=k}^\infty \delta_i^{\frac{2}{2-q}}$.
We deduce that
\begin{equation*}
\begin{aligned}
u_k &\le c(|F(x_k)-F(\bar{x})|+A_k)^{1-\theta} \\
    & = c\big\{c(1-\theta)[\varrho'(|F(x_k)-F(\bar{x})|+A_k)]^{-1}\big\}^{\frac{1-\theta}{\theta}} \\
    &\le c^{\frac{1}{\theta}}(1-\theta)^{\frac{1-\theta}{\theta}}
    \big\{[\varrho'(|F(x_k)-F(\bar{x})|)]^{-1}+[\varrho'(A_k)]^{-1}\big\}^{\frac{1-\theta}{\theta}},
\end{aligned}
\end{equation*}
where the second line follows from the special structure of the desingularizing function $\varrho$,
and the third line follows from the quasi-additivity-type property of $\varrho$ with $C_\varrho=1$.
Then, using \eqref{eq:KL-estimate-2}, we obtain
\begin{equation}\label{eq:lem-estimate}
  u_k \le c^{\frac{1}{\theta}}(1-\theta)^{\frac{1-\theta}{\theta}}
    \big\{L_c\|\eta_{x_{k-1}}^*\| + L_tC_q \delta_{k-1}^{\frac{2}{2-q}}+[\varrho'(A_k)]^{-1}\big\}^{\frac{1-\theta}{\theta}}.
\end{equation}
Next, we prove \eqref{eq:true-recurrence} by considering two cases.

Case 1: $\theta \in [\frac{1}{2},1)$.

We have $0<\frac{1-\theta}{\theta}\le 1$.
By \eqref{eq:lem-estimate} and the inequality
$(a+b+c)^{\frac{1-\theta}{\theta}}\le a^{\frac{1-\theta}{\theta}} + b^{\frac{1-\theta}{\theta}} +c^{\frac{1-\theta}{\theta}}$ which holds for all $a,b,c\ge0$, we obtain
\begin{equation}\label{}
  u_k \le c^{\frac{1}{\theta}}(1-\theta)^{\frac{1-\theta}{\theta}}
    \big\{L_c^{\frac{1-\theta}{\theta}}\|\eta_{x_{k-1}}^*\|^{\frac{1-\theta}{\theta}} + (L_tC_q)^{\frac{1-\theta}{\theta}} \delta_{k-1}^{\frac{2}{2-q}\cdot\frac{1-\theta}{\theta}}+[\varrho'(A_k)]^{-{\frac{1-\theta}{\theta}}}\big\}.
\end{equation}
Substituting the above estimate into \eqref{eq:recur-rewritten} and using the notation $D_k=\sum_{i=k}^\infty\|\eta_{x_i}^*\|$, we obtain
\begin{equation}\label{}
\begin{aligned}
D_k &\le (D_{k-1}-D_k) + C_1(D_{k-1}-D_k)^{\frac{1-\theta}{\theta}} + b_1\delta_{k-1}^{\frac{2}{2-q}\cdot\frac{1-\theta}{\theta}}
  + b_2 [\varrho'(A_k)]^{-{\frac{1-\theta}{\theta}}} \\
  &~~~~ + \sum_{i=k}^\infty\delta_{i-1}^{\frac{2}{2-q}}
+\sum_{i=k}^\infty[\varrho'(A_i)]^{-1},
\end{aligned}
\end{equation}
where $C_1:=\frac{1}{M}c^{\frac{1}{\theta}}[(1-\theta)L_c]^{\frac{1-\theta}{\theta}}$,
$b_1:=\frac{1}{M}c^{\frac{1}{\theta}}[(1-\theta)L_tC_q]^{\frac{1-\theta}{\theta}}$, and
$b_2:=\frac{1}{M}c^{\frac{1}{\theta}}(1-\theta)^{\frac{1-\theta}{\theta}}$.
Recall that $r = \frac{2p}{2-q}$. We then derive
\begin{equation*}
\begin{aligned}
 \delta_{k-1}^{\frac{2}{2-q}\cdot\frac{1-\theta}{\theta}} &=\frac{1}{(k+1)^{r\cdot\frac{1-\theta}{\theta}}},\\
 \sum_{i=k}^\infty \delta_{i-1}^{\frac{2}{2-q}} &= \sum_{i=k}^\infty\frac{1}{(i+1)^r}
    \le \int_{k}^\infty \frac{1}{x^r}dx =\mathcal{O}\left(\frac{1}{(k+1)^{r-1}}\right), \\
 [\varrho'(A_k)]^{-{\frac{1-\theta}{\theta}}} &= \mathcal{O}(A_k^{1-\theta})
 = \mathcal{O}\left(\frac{1}{(k+1)^{(r-1)(1-\theta)}}\right), \\
 \sum_{i=k}^\infty[\varrho'(A_i)]^{-1} &= \frac{1}{c(1-\theta)}\sum_{i=k}^\infty A_i^\theta
 = \mathcal{O}\left(\sum_{i=k}^\infty \frac{1}{(i+1)^{(r-1)\theta}}\right) = \mathcal{O}\left(\frac{1}{(k+1)^{(r-1)\theta-1}}\right).
\end{aligned}
\end{equation*}
Since $r\cdot\frac{1-\theta}{\theta}>(r-1)(1-\theta)$ and $r-1>(r-1)(1-\theta)$,
there exists a constant $C_2>0$ such that
\begin{equation}\label{}
  D_k \le (D_{k-1}-D_k) + C_1(D_{k-1}-D_k)^{\frac{1-\theta}{\theta}}
  + \frac{C_2}{(k+1)^{\beta}},
\end{equation}
where the exponent $\beta=\min\{(r-1)(1-\theta),(r-1)\theta-1\}$.

Case 2: $\theta \in (0,\frac{1}{2})$.

We have $\frac{1-\theta}{\theta}> 1$.
By \eqref{eq:lem-estimate} and the inequality
$(a+b+c)^{\frac{1-\theta}{\theta}}\le 3^{\frac{1-\theta}{\theta}}(a^{\frac{1-\theta}{\theta}}
+ b^{\frac{1-\theta}{\theta}} +c^{\frac{1-\theta}{\theta}})$ which holds for all $a,b,c\ge0$, we obtain
\begin{equation}\label{}
  u_k \le c^{\frac{1}{\theta}}[3(1-\theta)]^{\frac{1-\theta}{\theta}}
    \big\{L_c^{\frac{1-\theta}{\theta}}\|\eta_{x_{k-1}}^*\|^{\frac{1-\theta}{\theta}} + (L_tC_q)^{\frac{1-\theta}{\theta}} \delta_{k-1}^{\frac{2}{2-q}\cdot\frac{1-\theta}{\theta}}+[\varrho'(A_k)]^{-{\frac{1-\theta}{\theta}}}\big\}.
\end{equation}
The rest of the proof proceeds along the same lines as in Case 1.
This completes the proof of \eqref{eq:true-recurrence}.

Then, by \eqref{eq:true-recurrence} and Lemma \ref{lem:rate}, we obtain
\begin{equation*}
D_k =
\begin{cases}
\mathcal{O}(k^{-\beta}) & \text{if }\ 0<\theta\le \frac{1}{2}, \\
\mathcal{O}(k^{-\min\{\frac{1-\theta}{2\theta-1},\,\beta\}}) & \text{if }\ \frac{1}{2}<\theta<1,
\end{cases}
\end{equation*}
where $\beta=\min\{(r-1)(1-\theta),(r-1)\theta-1\}$.
Recall that $r>\frac{\theta+1}{\theta}$. A direct comparison yields:
\begin{itemize}
  \item If $0<\theta\le\frac{1}{2}$, then
  $(r-1)(1-\theta)>(r-1)\theta-1$.
  \item If $\frac{1}{2}<\theta<1$ and $\frac{\theta+1}{\theta}<r<\frac{2\theta}{2\theta-1}$, then
  $\frac{1-\theta}{2\theta-1}> (r-1)(1-\theta)> (r-1)\theta-1$.
  \item If $\frac{1}{2}<\theta<1$ and $r=\frac{2\theta}{2\theta-1}$, then
  $(r-1)(1-\theta)=(r-1)\theta-1=\frac{1-\theta}{2\theta-1}$.
  \item If $\frac{1}{2}<\theta<1$ and $r>\frac{2\theta}{2\theta-1}$, then
  $(r-1)\theta-1>(r-1)(1-\theta)>\frac{1-\theta}{2\theta-1}$.
\end{itemize}
We thus have
\begin{equation*}
D_k =
\begin{cases}
\mathcal{O}(k^{-[(r-1)\theta-1]}) & \text{if}\ 0<\theta\le \frac{1}{2},  \\
\mathcal{O}(k^{-[(r-1)\theta-1]}) & \text{if}\ \frac{1}{2}<\theta<1\ \text{and}\ \frac{\theta+1}{\theta}<r\le \frac{2\theta}{2\theta-1}, \\
\mathcal{O}(k^{-\frac{1-\theta}{2\theta-1}}) & \text{if}\ \frac{1}{2}<\theta<1\ \text{and}\ r> \frac{2\theta}{2\theta-1}.
\end{cases}
\end{equation*}

Finally, applying Lemma \ref{lem:Huang-Retr-estimate} and the triangle inequality yields
\begin{equation*}
 \mathrm{dist}(x_k,\bar{x}) \le \sum_{i=k}^\infty \mathrm{dist}(x_i,x_{i+1}) \le \kappa_0 \sum_{i=k}^\infty\|\eta_{x_i}^*\|=\kappa_0 D_k.
\end{equation*}
This completes the proof.
\end{proof}

\section{Convergence analysis of the Riemannian proximal gradient method with strong inexact oracle}
\setcounter{equation}{0}

\subsection{Strong inexact oracle for Riemannian optimization}

In this section, we analyze the convergence rate of the sequence of function values to the optimal value.
By adding an additional condition to Definition \ref{de-inexact-oracle}, we propose the following definition of a strong inexact oracle for Riemannian optimization.

\begin{de}\label{de-strong-inexact-oracle}
A retraction-convex function $f:\mathcal{M}\rightarrow \br$ is said to be equipped with a strong $(\delta,L)$-oracle of degree $q\in[0,2)$
if for all $x\in \mathcal{M}$, there exists $g_{\delta,L,q}(x)\in T_x\mathcal{M}$ such that
\begin{equation}
 0\le f(R_x(\eta)) - (f(x) + \langle g_{\delta,L,q}(x),\eta\rangle)\le \frac{L}{2}\|\eta\|^2 + \delta\|\eta\|^q,~~~~\forall\,\eta\in T_x\mathcal{M}.
\end{equation}
\end{de}

\noindent
\textbf{Remark:} In Definition \ref{de-strong-inexact-oracle},
the first-order information $g_{\delta,L,q}(x)$ is a Riemannian subgradient of $f$ at $x$.

For any pair of sufficiently close points $x,y\in \mathcal{M}$ with $x\neq y$,
we aim to measure the difference between $g_{\delta,L,q}(x)$ and $g_{\delta,L,q}(y)$.
However, since these two vectors lie in different tangent spaces,
we instead measure the difference between $\mathcal{T}_{\xi_y}^\sharp g_{\delta,L,q}(x)$ and $g_{\delta,L,q}(y)$.

\begin{thm}
Suppose $f:\mathcal{M}\rightarrow \br$ is a retraction-convex function equipped with a strong $(\delta,L)$-oracle of degree $q\in[0,2)$.
For any pair of sufficiently close points $x,y\in \mathcal{M}$, we denote $\xi_y = R_y^{-1}(x)\in T_y\mathcal{M}$ and $\zeta_x = R_x^{-1}(y)\in T_x\mathcal{M}$.
Then for any parameter $\rho>0$, the following inequality holds:
\begin{equation}\label{eq:thm-property}
\begin{split}
  &~~~~\|\mathcal{T}_{\xi_y}^\sharp g_{\delta,L,q}(x) - g_{\delta,L,q}(y)\|^2 \\
  &\leq 2(L+q\rho)\langle  g_{\delta,L,q}(x),\mathcal{T}_{\xi_y}\xi_y + \zeta_x\rangle
  +(L+q\rho)^2 \|\zeta_x\|^2 + \frac{2(2-q)(L+q\rho)}{\rho^{\frac{q}{2-q}}} \delta^{\frac{2}{2-q}}.
\end{split}
\end{equation}
\end{thm}
\begin{proof}
By a result from \cite{Hosseini-2011}, we have $\partial (f \circ R_y)(\xi_y) = [D R_y(\xi_y)]^\sharp \partial f(R_{y}(\xi_y))=\mathcal{T}_{\xi_y}^\sharp \partial f(x)$.
By Definition \ref{de-strong-inexact-oracle}, we have $g_{\delta,L,q}(x)\in \partial f(x)$.
Thus $\mathcal{T}_{\xi_y}^\sharp g_{\delta,L,q}(x)\in \partial (f \circ R_y)(\xi_y)$.

For any $\eta_y\in T_y \mathcal{M}$, the retraction-convexity of $f$ together with the identity $x=R_{y}(\xi_y)$ implies
\begin{equation*}
  f(R_y(\eta_y)) \ge  f(x) + \langle \mathcal{T}_{\xi_y}^\sharp g_{\delta,L,q}(x), \eta_y-\xi_y\rangle.
\end{equation*}
Applying \eqref{eq:inexact-inequality} to $f(R_y(\eta_y))$
and combining the resulting inequality with the above inequality, we obtain
\begin{equation*}
\begin{split}
 &~~~~ f(x) - \langle \mathcal{T}_{\xi_y}^\sharp g_{\delta,L,q}(x), \xi_y\rangle \\
 & \leq f(y) + \langle g_{\delta,L,q}(y)-\mathcal{T}_{\xi_y}^\sharp g_{\delta,L,q}(x), \eta_y \rangle
 + \frac{L + q\rho}{2} \|\eta_y\|^2 + \frac{2-q}{2\rho^{\frac{q}{2-q}}} \delta^{\frac{2}{2-q}}.
\end{split}
\end{equation*}
Substituting $\eta_y = -\frac{1}{L+q\rho}\big(g_{\delta,L,q}(y)-\mathcal{T}_{\xi_y}^\sharp g_{\delta,L,q}(x)\big)$ into the above inequality yields
\begin{equation*}
  f(x) - \langle \mathcal{T}_{\xi_y}^\sharp g_{\delta,L,q}(x), \xi_y\rangle
  \leq f(y)- \frac{1}{2(L+q\rho)}\|g_{\delta,L,q}(y)-\mathcal{T}_{\xi_y}^\sharp g_{\delta,L,q}(x)\|^2
  + \frac{2-q}{2\rho^{\frac{q}{2-q}}} \delta^{\frac{2}{2-q}}.
\end{equation*}
Finally, applying \eqref{eq:inexact-inequality} to $f(y)$ (recalling $y=R_x(\zeta_x)$)
and combining the resulting inequality with the above inequality, we arrive at inequality \eqref{eq:thm-property}.
\end{proof}

Furthermore, it can be shown that $\mathcal{T}_{\xi_y}\xi_y + \zeta_x$ is a higher-order infinitesimal with respect to $\|\xi_y\|$.
In particular, if the retraction $R$ reduces to the exponential map $\Exp$, then $\mathcal{T}_{\xi_y}\xi_y + \zeta_x\equiv 0$.
Consequently, we obtain the following corollary.

\begin{coro}
Suppose $f:\mathcal{M}\rightarrow \br$ is a retraction-convex function equipped with a strong $(\delta,L)$-oracle of degree $q\in [0,2)$.
For any pair of sufficiently close points $x,y\in \mathcal{M}$, we denote $\xi_y = R_y^{-1}(x)\in T_y\mathcal{M}$ and $\zeta_x = R_x^{-1}(y)\in T_x\mathcal{M}$. Then we have
\begin{equation}\label{eq:coro-property}
\|\mathcal{T}_{\xi_y}^\sharp g_{\delta,L,q}(x) - g_{\delta,L,q}(y)\|^2
= \mathcal{O}\left(\|\zeta_x\|^2+\delta^{\frac{2}{2-q}}\right) + o(\|g_{\delta,L,q}(x)\|\cdot\|\xi_y\|).
\end{equation}
In particular, if the retraction reduces to the exponential map, i.e., $R=\Exp$, we have
\begin{equation}\label{eq:coro-property-exp}
  \|\mathcal{T}_{\xi_y}^\sharp g_{\delta,L,q}(x) - g_{\delta,L,q}(y)\|^2
= \mathcal{O}\left(\|\zeta_x\|^2+\delta^{\frac{2}{2-q}}\right).
\end{equation}
\end{coro}
\begin{proof}
We define $F:=R_x^{-1}\circ R_y$. We have
\begin{equation*}
\begin{split}
& F(\xi_y) = R_x^{-1}(R_y(\xi_y))=R_x^{-1}(x)=0_x,\\
& F(0_y) = R_x^{-1}(R_y(0_y))=R_x^{-1}(y)=\zeta_x.
\end{split}
\end{equation*}
Differentiating $F=R_x^{-1}\circ R_y$ at $\xi_y$, we obtain
\begin{equation*}
  DF(\xi_y)= DR_x^{-1}(R_y(\xi_y))\circ DR_y(\xi_y) = DR_x^{-1}(x)\circ\mathcal{T}_{\xi_y}.
\end{equation*}
Differentiating the identity $R_x\circ R_x^{-1}=\id$ at $x$, we obtain
\begin{equation*}
  DR_x(R_x^{-1}(x))\circ DR_x^{-1}(x)=\id.
\end{equation*}
By the definition of the retraction, we have $DR_x(R_x^{-1}(x))=DR_x(0_x)=\id$.
Therefore, $DR_x^{-1}(x)=\id$. Thus, we obtain
\begin{equation*}
  DF(\xi_y) = \mathcal{T}_{\xi_y}.
\end{equation*}

We expand $F$ in a Taylor series around the point $\xi_y$:
\begin{equation*}
  F(\eta) = F(\xi_y) + DF(\xi_y)[\eta-\xi_y] + o(\|\eta-\xi_y\|).
\end{equation*}
Substituting $\eta=0_y$ into the above equation, we obtain
\begin{equation*}
  \zeta_x = 0_x + \mathcal{T}_{\xi_y}(-\xi_y) + o(\|\xi_y\|).
\end{equation*}
Then, using \eqref{eq:thm-property}, we obtain \eqref{eq:coro-property}.

In what follows, we assume the retraction is the exponential map, i.e., $R=\Exp$.

We define $\gamma(t):=\Exp_x(t\zeta_x)$ for $t\in[0,1]$.
Then $\gamma(0)=x$ and $\gamma(1)=y$, which means $\gamma(t)$ is a geodesic joining $x$ to $y$.
Moreover, we have $\gamma'(0)=\zeta_x$.

Similarly, we define $\tilde{\gamma}(t):=\Exp_y(t\xi_y)$ for $t\in[0,1]$.
Then $\tilde{\gamma}(0)=y$ and $\tilde{\gamma}(1)=x$, which means $\tilde{\gamma}(t)$ is a geodesic joining $y$ to $x$.
Moreover, we have $\tilde{\gamma}'(1)=D\Exp_y(\xi_y)[\xi_y]=\mathcal{T}_{\xi_y}\xi_y$.

It is clear that $\gamma(t)$ and $\tilde{\gamma}(t)$ are the same geodesic traversed in opposite directions, i.e., $\tilde{\gamma}(t)=\gamma(1-t)$.
Thus, $\tilde{\gamma}'(1)=-\gamma'(0)$, i.e., $\mathcal{T}_{\xi_y}\xi_y=-\zeta_x$. Then, using \eqref{eq:thm-property}, we obtain \eqref{eq:coro-property-exp}.
\end{proof}

\subsection{Convergence rate of the sequence of function values to the optimal value}

It is well known that in the Euclidean setting, the proximal gradient method \eqref{eq:PGM} has an
$\mathcal{O}(1/k)$ convergence rate of the sequence of function values to the optimal value for convex problems \cite{Beck-bk2017}.
To establish the corresponding convergence rate in the Riemannian setting, we require
the following assumption.

\begin{assump}\label{assump:Huang}
\cite{Huang-2022} There exists a constant $\kappa>0$ such that for any $x,y,z\in \Omega$, there holds
\begin{equation}\label{eq:assump-Huang}
  \left|\|R_x^{-1}(z)-R_x^{-1}(y)\|^2 - \|R_y^{-1}(z)\|^2\right| \le \kappa \|R_x^{-1}(y)\|^2,
\end{equation}
where $\Omega$ is a set containing the iterates $\{x_k\}$.
\end{assump}

The following theorem establishes the convergence rate of the sequence of function values for Algorithm \ref{alg:algorithm-IRPG}.
The proof mainly follows that in \cite[Theorem 2]{Huang-2022}. We omit the details here.

\begin{thm}\label{thm:function-convergence}
Suppose that $f$ is a retraction-convex function equipped with a strong $(\delta_k, L_k)$-oracle of degree $q \in [0,2)$ at each iteration $k$,
and that $h$ is a retraction-convex function.
Let $\{x_k\}$ be the sequence generated by Algorithm \ref{alg:algorithm-IRPG}.
Suppose further that Assumption \ref{assump:Huang} holds.
If there exist constants $\rho > 0$ and $t_{\min}>0$ such that $t_{\min} \leq t_k \leq \frac{1}{L_k + q\rho}$ for all $k\ge 0$,
then for all $n\geq 1$, we have
\begin{equation}\label{eq:function-convergence}
\begin{aligned}
  &~~~\min_{1\leq k \leq n}\{F(x_k)-F(x^*)\} \\
  &\le\frac{1}{2nt_{\min}}\|R_{x_0}^{-1}(x^*)\|^2 + \frac{\kappa}{n}[F(x_0) - F(x^*)] + \frac{(\kappa+1)(2-q)}{2n\rho^{\frac{q}{2-q}}} \sum_{k=0}^{n-1}\delta_k^{\frac{2}{2-q}},
\end{aligned}
\end{equation}
where $x^*$ denotes a global minimizer of $F$ on $\mathcal{M}$.
\end{thm}

In particular, when the oracle errors decay as $\delta_k = \frac{1}{(k+2)^p}$ for some $p>0$,
we have the following corollary.

\begin{coro}
Suppose the conditions of Theorem \ref{thm:function-convergence} hold.
Assume further that $\delta_k = \frac{1}{(k+2)^p}$ for some $p>0$. Then
\begin{equation*}
\min_{1\leq k \leq n}\{F(x_k)-F(x^*)\} =
\begin{cases}
\mathcal{O}(n^{-r}) & \text{if }\ 0<r<1, \\
\mathcal{O}(n^{-1}\ln n) & \text{if }\ r=1, \\
\mathcal{O}(n^{-1}) & \text{if }\ r> 1,
\end{cases}
\end{equation*}
where $r=\frac{2p}{2-q}$.
\end{coro}
\begin{proof}
If $0<r<1$, we have
\begin{equation*}
  \sum_{k=0}^{n-1}\delta_k^{\frac{2}{2-q}} = \sum_{k=0}^{n-1}\frac{1}{(k+2)^r}\le \int_{1}^{n+1}\frac{1}{x^r}dx
  \le \frac{(n+1)^{1-r}}{1-r}\le \frac{2^{1-r}}{1-r}n^{1-r}.
\end{equation*}
If $r=1$, we have
\begin{equation*}
  \sum_{k=0}^{n-1}\delta_k^{\frac{2}{2-q}} = \sum_{k=0}^{n-1}\frac{1}{k+2}\le \int_{1}^{n+1}\frac{1}{x}dx
  = \ln(n+1)\le 2\ln n,~~\forall\,n\geq 2.
\end{equation*}
If $r>1$, we have
\begin{equation*}
  \sum_{k=0}^{n-1}\delta_k^{\frac{2}{2-q}} = \sum_{k=0}^{n-1}\frac{1}{(k+2)^r} \leq \sum_{k=0}^{\infty}\frac{1}{(k+2)^r}<\infty.
\end{equation*}
Combining the three cases above with \eqref{eq:function-convergence} yields the desired results. 
\end{proof}

\section{Conclusion}
\setcounter{equation}{0}

In this paper, we define an inexact first-order oracle for Riemannian optimization
and conduct a rigorous convergence analysis for the Riemannian proximal gradient method equipped with this oracle, referred to as RPG-IO.
Under mild conditions on the oracle errors, we establish the global convergence of RPG-IO.
Under the additional assumption of the Riemannian KL property, we prove that the full sequence of iterates converges to a single stationary point, and derive explicit convergence rates when the KL exponent is specified. Under a strong inexact oracle, we further derive the convergence rate of the sequence of function values to the optimal value.

Overall, our results extend the classical proximal gradient convergence theory to the Riemannian setting with inexact first-order information, and provide a rigorous theoretical foundation for the design and implementation of practical Riemannian optimization algorithms in the inexact oracle regime.
A natural direction for future research is to extend the RPG-IO framework
to stochastic settings and to incorporate line search stepsize strategies.

\section*{Acknowledgement}

This research was partially supported by National Natural Science Foundation
of China (Grant Nos. 12171488, 62573298), Guangdong Provincial Key Laboratory (Grant No. 2023B1212060076),
Guangdong Basic and Applied Basic Research Foundation (Grant No. 2024A1515010988).

\section*{References}

\bibliography{UTF8-refs}

\begin{thebibliography}{10}
\expandafter\ifx\csname url\endcsname\relax
  \def\url#1{\texttt{#1}}\fi
\expandafter\ifx\csname urlprefix\endcsname\relax\def\urlprefix{URL }\fi
\expandafter\ifx\csname href\endcsname\relax
  \def\href#1#2{#2} \def\path#1{#1}\fi

\bibitem{Absil-bk2009}
P.~Absil, R.~Mahony, R.~Sepulchre, Optimization Algorithms on Matrix Manifolds,
  Princeton University Press, Princeton, 2009.

\bibitem{Boumal-bk2023}
N.~Boumal, An Introduction to Optimization on Smooth Manifolds, Cambridge
  University Press, Cambridge, 2023.

\bibitem{Sato-bk2021}
H.~Sato, Riemannian Optimization and Its Applications, SpringerBriefs in
  Control, Automation and Robotics, Springer, Cham, Switzerland, 2021.

\bibitem{Wen-2020}
J.~Hu, X.~Liu, Z.-W. Wen, Y.-X. Yuan, A brief introduction to manifold
  optimization, J. Oper. Res. Soc. China 8~(2) (2020) 199--248.

\bibitem{Ma-2020}
S.~Chen, S.~Ma, A.~M.-C. So, T.~Zhang, Proximal gradient method for nonsmooth
  optimization over the stiefel manifold, SIAM J. Optim. 30~(1) (2020)
  210--239.

\bibitem{Beck-2009}
A.~Beck, M.~Teboulle, A fast iterative shrinkage-thresholding algorithm for
  linear inverse problems, SIAM J. Imaging Sci. 2~(1) (2009) 183–--202.

\bibitem{Beck-bk2017}
A.~Beck, First-Order Methods in Optimization, MOS-SIAM Series on Optimization,
  Society for Industrial and Applied Mathematics, Philadelphia, PA, 2017.

\bibitem{Huang-2022-2}
W.~Huang, K.~Wei, An extension of fast iterative shrinkage-thresholding
  algorithm to {R}iemannian optimization for sparse principal component
  analysis, Numer. Linear Algebra Appl. 29 (2022) e2409.

\bibitem{Huang-2022}
W.~Huang, K.~Wei, Riemannian proximal gradient methods, Math. Program. 194
  (2022) 371--413.

\bibitem{Choi-2025}
W.~Choi, C.~Chun, Y.~M. Jung, S.~Yun, On the linear convergence rate of
  {R}iemannian proximal gradient method, Optim. Lett. 19 (2025) 667--687.

\bibitem{Beck-2023}
A.~Beck, I.~Rosset, A dynamic smoothing technique for a class of nonsmooth
  optimization problems on manifolds, SIAM J. Optim. 33~(3) (2023) 1473--1493.

\bibitem{Bai-2023}
F.~Bai, A.~Bartoli, The proxy step-size technique for regularized optimization
  on the sphere manifold, IEEE Trans. Pattern Anal. Mach. Intell. 45~(5) (2023)
  6428--6444.

\bibitem{Nesterov-2014}
O.~Devolder, F.~Glineur, Y.~Nesterov, First-order methods of smooth convex
  optimization with inexact oracle, Math. Program. 146 (2014) 37--75.

\bibitem{Nabou-2025}
Y.~Nabou, F.~Glineur, I.~Necoara, Proximal gradient methods with inexact oracle
  of degree $q$ for composite optimization, Optimization Letters 19 (2025)
  285–--306.

\bibitem{Huang-2023}
W.~Huang, K.~Wei, An inexact {R}iemannian proximal gradient method, Comput.
  Optim. Appl. 85 (2023) 1--32.

\bibitem{Deng-2023}
K.~Deng, Z.~Peng, A manifold inexact augmented {L}agrangian method for
  nonsmooth optimization on {R}iemannian submanifolds in {E}uclidean space, IMA
  J. Numer. Anal. 43 (2023) 1653--1684.

\bibitem{Deng-2025-2}
K.~Deng, J.~Hu, J.~Wu, Z.~Wen, Oracle complexities of augmented lagrangian
  methods for nonsmooth composite optimization on a compact submanifold, Math.
  Oper. Res. (2025) 1--27.

\bibitem{Deng-2025}
J.~Zhou, K.~Deng, H.~Wang, Z.~Peng, Inexact {R}iemannian gradient descent
  method for nonconvex optimization with strong convergence, J. Sci. Comput.
  103 (2025).

\bibitem{Bolte-2009}
H.~Attouch, J.~Bolte, On the convergence of the proximal algorithm for
  nonsmooth functions involving analytic features, Math. Program. 116 (2009)
  5--16.

\bibitem{Li-2023}
X.~Li, A.~Milzarek, J.~Qiu, Convergence of random reshuffling under the
  {K}urdyka--{${\L}$}ojasiewicz inequality, SIAM J. Opt. 33~(2) (2023)
  1092--1120.

\bibitem{Lee-book2018}
J.~M. Lee, Introduction to Riemannian Manifolds, 2nd Edition, Vol. 176 of
  Graduate Texts in Mathematics, Springer, Switzerland, 2018.

\bibitem{Carmo-book1992}
M.~P. do~Carmo, Riemannian Geometry, Mathematics: Theory \& Applications,
  Birkh{\"a}user Boston, Inc., Boston, MA, 1992.

\bibitem{Clarke-book1990}
F.~H. Clarke, Optimization and Nonsmooth Analysis, Vol.~5 of Classics in
  Applied Mathematics, Society for Industrial and Applied Mathematics,
  Philadelphia, PA, 1990.

\bibitem{Hosseini-2018}
S.~Hosseini, W.~Huang, R.~Yousefpour, Line search algorithms for locally
  lipschitz functions on {R}iemannian manifolds, SIAM J. Optim. 28~(1) (2018)
  596–--619.

\bibitem{Hosseini-2011}
S.~Hosseini, M.~Pouryayevali, Generalized gradients and characterization of
  epi-{L}ipschitz sets in {R}iemannian manifolds, Nonlinear Anal. Theory
  Methods Appl. 74~(12) (2011) 3884--3895.

\bibitem{Yang-2014}
W.~H. Yang, L.-H. Zhang, R.~Song, Optimality conditions for the nonlinear
  programming problems on {R}iemannian manifolds, Pac. J. Optim. 10~(2) (2014)
  415--434.

\bibitem{Boumal-2018}
N.~Boumal, P.-A. Absil, C.~Cartis, Global rates of convergence for nonconvex
  optimization on manifolds, IMA J. Numer. Anal. 39~(1) (2018) 1--33.

\bibitem{Bolte-2007}
J.~Bolte, A.~Daniilidis, A.~Lewis, The {{\L}}ojasiewicz inequality for
  nonsmooth subanalytic functions with applications to subgradient dynamical
  systems, SIAM J. Optim. 17~(4) (2007) 1205--1223.

\bibitem{Bolte-2013}
H.~Attouch, J.~Bolte, B.~F. Svaiter, Convergence of descent methods for
  semi-algebraic and tame problems: proximal algorithms, forward-backward
  splitting, and regularized {Gauss-Seidel} methods, Math. Program. 137 (2013)
  91--129.

\bibitem{Bolte-2014}
J.~Bolte, S.~Sabach, M.~Teboulle, Proximal alternating linearized minimization
  for nonconvex and nonsmooth problems, Math. Program. 146~(1--2) (2014)
  459--494.

\end{thebibliography}
\bibliographystyle{elsarticle-num}

\end{document}